\UseRawInputEncoding
\documentclass[12pt, reqno]{amsart}
\usepackage[margin=1in]{geometry}
\usepackage{amssymb,latexsym,amsmath,amscd,amsfonts}
\usepackage{latexsym}
\usepackage[mathscr]{eucal}
\usepackage{bm}
\usepackage{graphicx}
\usepackage{forest}
\usepackage{tikz-qtree}
\usepackage{mathptmx}
\usepackage{amssymb}
\usepackage{upgreek}
\usepackage{amsthm}
\usepackage{dcolumn}
\usepackage{tikz-cd}
\usepackage[all]{xy}
\usepackage{enumitem}
\usepackage[utf8]{inputenc}

\def \qed {\hfill \vrule height6pt width 6pt depth 0pt}
\def\textmatrix#1&#2\\#3&#4\\{\bigl({#1 \atop #3}\ {#2 \atop #4}\bigr)}
\def\dispmatrix#1&#2\\#3&#4\\{\left({#1 \atop #3}\ {#2 \atop #4}\right)}
\newcommand{\beg}{\begin{equation}}
	\newcommand{\eeg}{\end{equation}}
\newcommand{\ben}{\begin{eqnarray*}}
	\newcommand{\een}{\end{eqnarray*}}

\newtheorem{thm}{Theorem}[section]
\newtheorem{cor}[thm]{Corollary}
\newtheorem{lem}[thm]{Lemma}

\newtheorem{prop}[thm]{Proposition}
\numberwithin{equation}{section} \theoremstyle{definition}
\newtheorem{defn}[thm]{Definition}

\newcommand{\C}{\mathbb{C}}

\newcommand{\D}{\mathbb{D}}

\newcommand{\T}{\mathbb{T}}
\newcommand{\N}{\mathbb{N}}
\newcommand{\Z}{\mathbb{Z}}

\newcommand{\HS}{\mathcal{H}}

\def\textmatrix#1&#2\\#3&#4\\{\bigl({#1 \atop #3}\ {#2 \atop #4}\bigr)}
\def\dispmatrix#1&#2\\#3&#4\\{\left({#1 \atop #3}\ {#2 \atop #4}\right)}

\title[Regular $q$-unitary dilation, Brehmer's positivity and von Neumann's inequality]{Theory of $q$-commuting contractions-II: Regular dilation, Brehmer's positivity and von Neumann's inequality}

\author[PAL, SAHASRABUDDHE AND TOMAR]{SOURAV PAL, PRAJAKTA SAHASRABUDDHE AND NITIN TOMAR}

\address[Sourav Pal]{Mathematics Department, Indian Institute of Technology Bombay,
	Powai, Mumbai - 400076, India.} \email{sourav@math.iitb.ac.in, souravmaths@gmail.com}

\address[Prajakta Sahasrabuddhe]{Mathematics Department, Indian Institute of Technology Bombay, Powai, Mumbai-400076, India.} \email{praju1093@gmail.com}

\address[Nitin Tomar]{Mathematics Department, Indian Institute of Technology Bombay, Powai, Mumbai-400076, India.} \email{tomarnitin414@gmail.com}



\begin{document}
	
	\maketitle

	\begin{abstract}
		It is well-known that a commuting family of contractions possesses a regular unitary dilation if and only if it satisfies Brehmer's positivity condition. We extend this theorem to any family $\mathcal T$ of $q$-commuting contractions with $\|q\|=1$ by showing the equivalence of the following three statements: $(i)$ $\mathcal T$ admits a regular $q$-unitary dilation; $(ii)$ $\mathcal T$ satisfies Brehmer's positivity condition; $(iii)$ $\mathcal T$ admits a $Q$-unitary dilation for a family of $Q$-commuting unitaries. We achieve the first part of the result by an application of Stinespring's dilation theorem on a particular completely positive map acting on a quotient algebra of a group $C^*$-algebra, where the underlying group is a free group, and the second part is obtained by an application of Naimark's theorem. Next, we prove that $\mathcal{T}$ admits a regular $q$-unitary dilation in each of the following cases: $(i) \mathcal{T}$ consists of $q$-commuting isometries; $(ii) \mathcal{T}$ consists of doubly $q$-commuting contractions; $(iii) \mathcal{T}$ is a countable family on a Hilbert space $\mathcal H$ and $\sum_{\alpha \in \Lambda} \|T_{\alpha}h\|^2 \leq \|h\|^2$ for all $h \in \mathcal{H}$. An analogue of von Neumann's inequality is obtained for these classes of $q$-commuting contractions. Further, the main results are generalized to any family of $Q$-commuting contractions, where $Q$ consists of commuting unitaries.
	\end{abstract}

	\section{Introduction}\label{sec01}
	
	\vspace{0.2cm}
	
	\noindent Throughout the paper, all operators are bounded linear maps acting on complex Hilbert spaces. We denote by $\mathbb{C}, \mathbb{D}$ and $\mathbb{T}$ the complex plane, the unit disk and the unit circle in the complex plane, respectively with center at the origin. Given a Hilbert space $\HS$, the algebra of operators acting on $\HS$ is denoted by $\mathcal{B}(\HS)$ and the identity operator is denoted by $I_\HS$, or simply $I$ when no confusion arises. A contraction is an operator with norm at most $1$. 
		
\subsection{Motivation}	One of the most wonderful discoveries in operator theory is Bela Sz. Nagy's unitary dilation of a contraction \cite{NagyII}, which states that every contraction dilates to a unitary, i.e., given any contraction $T$ acting on a Hilbert space $\HS$, there is a Hilbert space $\mathcal K$ that contains $\HS$ as a closed linear subspace and a unitary $U$ acting on $\mathcal K$ such that 
\[
T^k=P_\HS U^k|_\HS \quad (k=0, 1, 2, \dotsc),
\]
where $P_\HS$ denotes the orthogonal projection of $\mathcal K$ onto $\HS$. The next appealing step to Bela Sz. Nagy's unitary dilation is Ando's dilation for a pair of commuting contractions \cite{Ando}, which asserts that any pair of commuting contractions $(T_1,T_2)$ acting on a Hilbert space $\HS$ admits dilation to a pair of commuting unitaries $(U_1,U_2)$ acting on a Hilbert space $\mathcal K \supseteq \HS$, i.e,
\[
T_1^{K_1}T_2^{k_2}= P_\HS U_1^{k_1}U_2^{k_2}|_\HS \quad (k_1,k_2=0, 1, 2, \dotsc).
\] 
However, Parrott \cite{Par} shows by a counter example that such a unitary dilation is not possible in general for a commuting tuple of contractions $(T_1,\dots, T_n)$ for $n\geq 3$. This leads to one of the most difficult and unsettled open problems in operator theory: what are all commuting $n$-tuples of contractions that possess unitary dilations when $n \geq 3$ ? Attempts have been made to characterize such $n$-tuples of contractions, though only a few special cases are known till date and the original problem remains unresolved, e.g., see \cite{Agler, Bal:Tim:Tre, Bal:Tre:Vin, Cur:Vas 1, Cur:Vas 2, SP1, SP2} or the classic \cite{Nagy} and the references therein. Later, Brehmer \cite{Brehmer} introduced the notion of regular unitary dilation for commuting contractions. Note that, a unitary dilation of a commuting tuple of contractions $(T_1, \dots, T_n)$ acting on $\HS$ is a commuting tuple of unitaries $(U_1, \dots , U_n)$ on a Hilbert space $\mathcal K \supseteq \HS$ satisfying
	\begin{equation} \label{eqn:new-001}
	T_1^{k_1}\dots T_n^{k_n} =P_{\HS}\, U_1^{k_1}\dots U_n^{k_n}|_{\HS}, \quad k_1, \dots , k_n \in \mathbb N \cup \{0\}.
	\end{equation}
In view of this, Brehmer's regular unitary dilation is somewhat stronger than the unitary dilation in the sense that it involves both $(T_1, \dots , T_n)$ and its adjoint $(T_1^*, \dots , T_n^*)$ in the dilation relation. For a commuting tuple of contractions $\underline{T}=(T_1, \dotsc, T_k)$ acting on a Hilbert space $\HS$ and for a tuple of positive integers $m=(m_1, \dotsc, m_k)$, the standard convention is to write $T^m=T_1^{m_1}\dotsc T_k^{m_k}$. Also, if $m=(m_1, \dotsc, m_k) \in \Z^k$, then $T(m)=(T^{m^-})^*T^{m^+}$ with $m^+=(\max\{m_1, 0\}, \dotsc, \max\{m_k, 0\})$ and $m^-=-(\min\{m_1, 0\}, \dotsc, \min\{m_k, 0\})$. The tuple $\underline{T}$ is said to have a \textit{regular unitary dilation} if there exist a Hilbert space $\mathcal{K}$ containing $\HS$ and a commuting tuple $\underline{U}=(U_1, \dotsc, U_k)$ of unitaries on $\mathcal{K}$ such that 
	\begin{equation}\label{eqn_reg}
		T(m)=P_\HS U(m)|_\HS  \quad \text{for all}  \ m \in \Z^k.
	\end{equation}
	Moreover, a commuting family of contractions $\mathcal{T}=\{T_\alpha : \alpha \in \Lambda\}$ of acting on $\HS$ is said to have a \textit{regular unitary dilation} if there exist a Hilbert space $\mathcal{K} \supseteq \HS$ and a commuting family $\mathcal{U}=\{U_\alpha: \alpha \in \Lambda\}$ of unitaries on $\mathcal{K}$ such that (\ref{eqn_reg}) holds for every finite tuple in $\mathcal{T}$. Unlike unitary dilations, there is a complete characterization due to Brehmer \cite{Brehmer} of the commuting families of contractions admitting regular unitary dilation.
		
	\begin{thm}[Brehmer, \cite{Brehmer}]\label{thm_reg_Bre}
		A  commuting family of contractions $\mathcal{T}=\{T_\alpha : \alpha \in \Lambda\}$ acting on a Hilbert space $\HS$ has a regular unitary dilation if and only if
		\begin{equation}\label{eqn_Bre}
			S(u)=	\underset{\{\alpha_1, \dotsc, \alpha_k\} \subset u}{\sum}(-1)^k(T_{\alpha_1}\dotsc T_{\alpha_k})^*(T_{\alpha_1}\dotsc T_{\alpha_k}) \geq 0 \ \ \text{for every finite subset $u$ of $\Lambda$.}
		\end{equation}
	\end{thm}
The condition in (\ref{eqn_Bre}) is known as the \textit{Brehmer's positivity}. Later, Halperin \cite{Halp} provides an alternative proof to Brehmer's famous result. An application of Brehmer's theorem and a fine observation due to Attele and Lubin (see Proposition 2 in \cite{Attele}) provide the following classes of contractions possessing regular unitary dilation by satisfying Brehmer's positivity condition.
	
	\begin{prop}[\cite{Nagy}, CH-I, Proposition 9.2 \& \cite{Attele}, Proposition 2 ]\label{prop111}
		A commuting family of contractions $\mathcal{T}=\{T_\alpha : \alpha \in \Lambda\}$ acting on $\HS$ possesses a regular unitary dilation in each of the cases: 
		\begin{enumerate}
			\item[(i)] $\mathcal{T}$ consists of isometries;
			\item[(ii)] $\mathcal{T}$ consists of doubly commuting contractions;
			\item[(iii)] $\mathcal T$ is a countable family and $\sum_{\alpha \in \Lambda} \|T_{\alpha}h\|^2 \leq \|h\|^2$ for all $h \in \HS$.
		\end{enumerate} 
	\end{prop}
Evidently, regular unitary dilation implies unitary dilation and thus commuting tuple $(T_1, \dots, T_n)$ having regular unitary dilation must satisfy
$
	p(T_1, \dotsc, T_k)=P_\HS p(U_1, \dotsc, U_k)|_\HS
	$ 
	for every $p \in \C[z_1, \dotsc, z_k]$. This gives an immediate von Neumann's inequality on the closed polydisc $\overline{\D}^n$ for such a tuple, i.e.,
	\[
		\|p(T_1, \dotsc, T_k)\| \leq \|p\|_{\infty, \T^k}
		\]
		for every polynomial $p$ in $\C[z_1, \dotsc, z_k]$. Hence, von Neumann's inequality holds for each of the classes described in Proposition \ref{prop111}.
		
\subsection{The main results of the paper} In this article, we generalize Theorem \ref{thm_reg_Bre} and Proposition \ref{prop111} to a $q$-commuting family of contractions. Consequently, we obtain a von Neumann type inequality for such a $q$-commuting family. A pair of operators $T_1,T_2$ on a Hilbert space $\HS$ is said to be $q$-\textit{commuting} for a scalar $q$ if $T_1T_2=qT_2T_1$. The definition of more general $q$-commuting family of operators is given below.

\begin{defn}
		A family of operators $\mathcal{T}=\{T_\alpha: \alpha \in \Lambda \}$ acting on a Hilbert space $\HS$ is said to be \textit{$q$-commuting} for a family of non-zero complex scalars $q=\{q_{\alpha \beta} : q_{\alpha\beta}=q_{\beta \alpha}^{-1}, \  \alpha, \beta \in \Lambda, \  \alpha \ne \beta \ \}$ if $T_\alpha T_\beta =q_{\alpha \beta}T_\beta T_\alpha$ for all $\alpha, \beta$ in $\Lambda$ with $\alpha \ne \beta$. If each $q_{\alpha \beta}$ is unimodular, then $\mathcal{T}$ is said to be a $q$-commuting family with $\|q\|=1$. Moreover, $\mathcal T$ is said be \textit{doubly $q$-commuting} if $T_\alpha T_\beta =q_{\alpha \beta}T_\beta T_\alpha$ and $T_\alpha T_\beta^*=\overline{q}_{\alpha \beta}T_\beta^*T_\alpha$ for all $\alpha, \beta$ in $\Lambda$ with $\alpha \ne \beta$.
	\end{defn}	 
	
A broader framework is provided by $Q$-commuting and doubly $Q$-commuting contractions, which generalize $q$-commuting contractions with $\|q\|=1$ and doubly $q$-commuting contractions respectively. 
	
	\begin{defn}
		A family $\{T_\alpha: \alpha \in \Lambda\}$ of operators on a Hilbert space $\HS$ is said be  \textit{$Q$-commuting} for a family of commuting unitaries $Q=\{Q_{\alpha \beta} \in \mathcal{B}(\HS) :  Q_{\alpha \beta}=Q_{\beta \alpha}^*,  \ \alpha \ne \beta \ \text{in} \ \Lambda \}$ if $T_\alpha T_\beta =Q_{\alpha \beta}T_\beta T_\alpha$ and $T_kQ_{\alpha \beta}=Q_{\alpha \beta}T_k$ for all $\alpha, \beta, k$ in $\Lambda$ with $\alpha \ne \beta$. In addition, if $T_\alpha T_\beta^*=Q_{\alpha \beta}^*T_\beta^*T_\alpha$, then $\mathcal{T}$ is said to be \textit{doubly $Q$-commuting}.
	\end{defn} 		
	
Dilation and lifting of $q$-commuting operators in two or more variables are well-studied, e.g., see \cite{Ball, Barik, Bis:Pal:Sah, Dey, DeyI, Jeu, K.M., Sebestyen} and references therein. Notably in \cite{Sebestyen}, Sebesty\'{e}n proved that an anticommuting pair of contractions (i.e., when $q=-1$) on a Hilbert space admits a dilation to an anticommuting pair of unitaries. Keshari and Mallick \cite{K.M.} generalized this result to any $q$-commuting pair of contractions with $\|q\|=1$. 	In the multivariable setting, it was proved in \cite{Barik} (see Theorem 3.10 in \cite{Barik}) that any $q$-commuting tuple $\underline{T}=(T_1, \dotsc, T_n)$ of contractions with $\|q\|=1$ admits a dilation to $q$-commuting $n$-tuple of isometries if $\underline{T}$ satisfies
	\begin{equation}\label{Szego}
		\underset{\{\alpha_1, \dotsc, \alpha_k\} \subset u}{\sum}(-1)^k(T_{\alpha_1}\dotsc T_{\alpha_k})(T_{\alpha_1}\dotsc T_{\alpha_k})^* \geq 0 \quad \text{for every $u \subseteq \{1, ,\dotsc, n\}$}.	
	\end{equation}
We generalize this result from \cite{Barik} in Corollary \ref{cor_403} and prove that a $q$-commuting family of contractions $\mathcal{T}=\{T_\alpha : \alpha \in \Lambda\}$ with $\|q\|=1$ satisfying (\ref{Szego}) for every finite subset $u$ of $\Lambda$ can be dilated to a $q$-commuting family of unitaries. Also, it was proved in \cite{Ball} that a $q$-commuting tuple of isometries with $\|q\|=1$ admits an extension to a $q$-commuting tuple of unitaries. Here we generalize this in Corollary \ref{cor_q_iso_ext} by proving that any $q$-commuting family of isometries with $\|q\|=1$ admits an extension to a $q$-commuting family of unitaries with $\|q\|=1$. Thus, the problem of finding a $q$-unitary dilation of a $q$-commuting family of contractions can be resolved by obtaining a $q$-isometric dilation for the family. First we define regular unitary dilation for a $q$-commuting family of contractions in the following canonical way.
\begin{defn}\label{defn_reg_q}
		Let $\mathcal{T}=\{T_\alpha : \alpha \in \Lambda\}$ be a $q$-commuting family of contractions  with $\|q\|=1$ acting on a Hilbert space $\mathcal{H}$, where $(\Lambda, \preceq)$ is a well-ordered set. We say that $\mathcal{T}$ admits a \textit{regular $Q$-unitary dilation} if there exist a Hilbert space $\mathcal
		{K} \supseteq \HS$ and a $Q$-commuting family $\mathcal{U}=\{U_\alpha : \alpha\in \Lambda\}$ of unitaries acting on $\mathcal{K}$ such that $Q_{\alpha \beta}|_{\HS}=q_{\alpha \beta}I_\HS$ for all $\alpha, \beta \in \Lambda$ with $\alpha \ne \beta$ and 
		\begin{align}\label{eqn**}
			\underset{1\leq i<j \leq k}\prod q_{\alpha_{i}\alpha_{j}}^{-m_{\alpha_{i}}^+m_{\alpha_{j}}^-}\left[(T_{\alpha_{1}}^{m_{\alpha_{1}}^-})^*\dotsc (T_{\alpha_{k}}^{m_{\alpha_{k}}^-})^*\right]\left[T_{\alpha_{1}}^{m_{\alpha_{1}}^+}\dotsc T_{\alpha_{k}}^{m_{\alpha_{k}}^+}\right]=P_\mathcal{H}U_{\alpha_{1}}^{m_{\alpha_{1}}}\dotsc U_{\alpha_{k}}^{m_{\alpha_{k}}}|_\mathcal{H}
		\end{align}
		for $m_{\alpha_1}, \dotsc, m_{\alpha_k} \in \mathbb{Z}$ and $\alpha_1, \dotsc, \alpha_k \in \Lambda$ with $\alpha_1 \preceq \dotsc \preceq \alpha_k$.	In addition, if the family 
		$
		Q=\left\{\widetilde{q}_{\alpha \beta}I_{\mathcal{K}} :\widetilde{q}_{\alpha \beta} \in \widetilde{q} \right\}
		$ for  
		$
		\widetilde{q}=\left\{\widetilde{q}_{\alpha \beta} \in \T \ : \ \widetilde{q}_{\alpha\beta}=\widetilde{q}_{\beta \alpha}^{-1} , \ \alpha \ne \beta \ \ \text{and} \ \ \alpha, \beta \in \Lambda \right\}
		$,
		then we say that $\mathcal{T}$ has a \textit{regular $\widetilde{q}$-unitary} dilation. 
	\end{defn}
	
In Section \ref{sec09},	we further discuss the motivation behind the above definition. The assumption that $\Lambda$ is well-ordered in the above definition may seem redundant. However, the purpose of mentioning an order on $\Lambda$ is to emphasize on the order of operators $T_{\alpha_1}, \dotsc, T_{\alpha_k}$ and $U_{\alpha_1}, \dotsc, U_{\alpha_k}$ appearing in \eqref{eqn**}. Putting an order on $\Lambda$ is crucial since we are dealing with operators in a non-commutative setting. Also, if the families $\mathcal{T}$ and $\mathcal{U}$ as in Definition \ref{defn_reg_q} follow the same $q$-intertwining relations, then \eqref{eqn**} is independent of the order in which $\alpha_1, \dotsc, \alpha_k$ appear. In this case, the definition coincides with that of a regular $q$-unitary dilation of a $q$-commuting family of contractions which is given below.
\begin{defn}\label{defn_reg_qI} 
		A $q$-commuting family $\mathcal{T}=\{T_\alpha : \alpha \in \Lambda\}$ of contractions with $\|q\|=1$ is said to have a \textit{regular $q$-unitary dilation} if there exist a Hilbert space $\mathcal
		{K} \supseteq \mathcal{H}$ and a $q$-commuting family $\mathcal{U}=\{U_\alpha : \alpha \in \Lambda \}$ of unitaries acting on $\mathcal{K}$ such that
		\[
		\underset{1\leq i<j \leq k}\prod q_{\alpha_{i}\alpha_{j}}^{-m_{\alpha_{i}}^+m_{\alpha_{j}}^-}\left[(T_{\alpha_{1}}^{m_{\alpha_{1}}^-})^*\dotsc (T_{\alpha_{k}}^{m_{\alpha_{k}}^-})^*\right]\left[T_{\alpha_{1}}^{m_{\alpha_{1}}^+}\dotsc T_{\alpha_{k}}^{m_{\alpha_{k}}^+}\right]=P_\mathcal{H}U_{\alpha_{1}}^{m_{\alpha_{1}}}\dotsc U_{\alpha_{k}}^{m_{\alpha_{k}}}|_\mathcal{H}
		\]
		for every $m_{\alpha_1}, \dotsc, m_{\alpha_k} \in \mathbb{Z}$ and $\alpha_1, \dotsc, \alpha_k$ in $\Lambda$. 
	\end{defn}
	
 The following is the first main result of this article which generalizes Theorem \ref{thm_reg_Bre}. 
	
	\begin{thm}\label{thm912}
		Let $\mathcal{T}=\{T_\alpha : \alpha \in \Lambda\}$ be a $q$-commuting family of contractions  with $\|q\|=1$ acting on a Hilbert space $\HS$. Then the following are equivalent:
\begin{enumerate}
\item $\mathcal{T}$ admits a regular $q$-unitary dilation ;

\item $\mathcal T$ satisfies the Brehmer's positivity condition, i.e.,
\[
		S(u)= \underset{\{\alpha_1, \dotsc, \alpha_m\} \subset u}{\sum}(-1)^m(T_{\alpha_1}\dotsc T_{\alpha_m})^*(T_{\alpha_1}\dotsc T_{\alpha_m}) \geq 0
		\] 
		for every finite subset $u$ of $\Lambda$ ;
		
		\item $\mathcal T$ admits a regular $Q$-unitary dilation for a $Q$-commuting family of unitaries $\mathcal{U}=\{U_\alpha : \alpha\in \Lambda\}$.
\end{enumerate}		
	\end{thm}
	
The equivalence of $(1)$ and $(2)$ of this theorem is obtained by an application of Stinespring's dilation theorem on a particular completely positive map acting on a quotient algebra of a group $C^*$-algebra, where the underlying group is a free group. Also, $(2) \Leftrightarrow (3)$ is proved using Naimark's theorem. Next, we have the following theorem which is another main result of this article.
	
	\begin{thm}\label{dilation of DCV}
		Let $\mathcal{T}=\{T_\alpha : \alpha \in \Lambda\}$ be a $q$-commuting family of contractions  with $\|q\|=1$ on a Hilbert space $\HS$. Then $\mathcal{T}$ admits a regular $q$-unitary dilation in each of the following cases:
		\begin{enumerate}
			\item $\mathcal{T}$ consists of isometries;
			\item $\mathcal{T}$ consists of doubly $q$-commuting contractions;
			\item $\mathcal{T}$ is a countable family and $\sum_{\alpha \in \Lambda} \|T_{\alpha}h\|^2 \leq \|h\|^2$ for all $h \in \HS$. 
		\end{enumerate}
	\end{thm}
The proofs are long and involve detailed computational steps. To make the algorithm transparent to the readers, we first prove Theorem \ref{thm912} and Theorem \ref{dilation of DCV} in Section \ref{sec09} for a finite family of $q$-commuting contractions. The proofs of the finite case pave the way to establish the general case in Section \ref{sec10}. We also obtain a von Neumann-type inequality for $q$-commuting contractions with $\|q\|=1$. However, in this case, one cannot choose the polynomial algebra since we do not have commutativity conditions anymore. The key step here is to identify the algebra of functions that can replace the polynomial algebra. Taking cue from the proof of Theorem \ref{regular q-unitary dilationII}, we achieve the desired inequality in Theorem \ref{thm_vNIII}. Further, we show that Theorems \ref{thm912} and \ref{dilation of DCV} hold in a more general setting of $Q$-commuting family, where $Q$ consists of commuting unitaries.
	
\section{Regular $q$-unitary dilation: The finite case}\label{sec09}
	
	\noindent In this Section, we prove Theorem \ref{thm912} and Theorem \ref{dilation of DCV} for a finite family of $q$-commuting contractions with $\|q\|=1$. Consequently, we have a von Neumann type inequality. Recall that a commuting tuple $(T_1, \dotsc, T_k)$ of contractions acting on a Hilbert space $\mathcal{H}$ is said to have a regular unitary dilation if there is a commuting tuple $(U_1, \dotsc, U_k)$ of unitaries acting on some Hilbert space $\mathcal{K}$ containing $\mathcal{H}$ such that
	\[
	\left[(T_1^{n_1^-})^*\dotsc (T_k^{n_k^-})^*\right]\left[T_1^{n_1^+} \dotsc T_k^{n_k^+}\right]=P_\mathcal{H}U_1^{n_1}\dotsc U_k^{n_k}|_\mathcal{H} \quad (n_1, \dotsc, n_k \in \mathbb{Z})
	\]
	or equivalently, 				
	\begin{align}\label{eqn401}
		\left[(T_1^{n_1^-})^*\dotsc (T_k^{n_k^-})^*\right]\left[T_1^{n_1^+} \dotsc T_k^{n_k^+}\right]=P_\mathcal{H}\left[(U_1^{n_1^-})^*\dotsc (U_k^{n_k^-})^*\right]\left[U_1^{n_1^+} \dotsc U_k^{n_k^+}\right]\bigg|_\mathcal{H} \quad (n_1, \dotsc, n_k \in \mathbb{Z}).
	\end{align}
Suppose $\underline{T}=(T_1, \dotsc, T_k)$ and $\underline{U}=(U_1, \dotsc, U_k)$ are $q$-commuting tuples of contractions and unitaries, respectively, with $\|q\|=1$. Evidently, the tuple $\underline{U}$ is doubly $q$-commuting. So, we have
	\begin{align} \label{eqn402}
		(U_1^{n_1^-})^*\dotsc (U_k^{n_k^-})^*U_1^{n_1^+} \dotsc U_k^{n_k^+}= \underset{1\leq i<j \leq k}\prod q_{ij}^{{n_i}^+n_{j}^-}(U_1^{n_1^-})^*U_1^{n_1^+} \dotsc (U_k^{n_k^-})^*U_k^{n_k^+}=\underset{1\leq i<j \leq k}\prod q_{ij}^{{n_i}^+n_{j}^-}U_1^{n_1} \dotsc U_k^{n_k}
	\end{align}
	for every $n_1, \dotsc, n_k$ in $\Z$. So, if \eqref{eqn401} holds for $\underline{T}$ and $\underline{U}$, then we have
	\[
	\underset{1\leq i<j \leq k}\prod q_{ij}^{-{n_i}^+n_{j}^-}\left[(T_1^{n_1^-})^*\dotsc (T_k^{n_k^-})^*\right]\left[T_1^{n_1^+} \dotsc T_k^{n_k^+}\right]=P_\mathcal{H}U_1^{n_1} \dotsc U_k^{n_k}|_\mathcal{H} \quad (n_1, \dotsc, n_k \in \mathbb{Z}).
	\]
This is the motivation behind Definition \ref{defn_reg_q}. As we have mentioned in the `Introduction' that if the families $\mathcal{T}$ and $\mathcal{U}$ as in Definition \ref{defn_reg_q} follow the same $q$-intertwining relations, then regular $q$-unitary dilation of a $q$-commuting family as in Definition \ref{defn_reg_q} coincides with that in Definition \ref{defn_reg_qI}. To see this, let $\mathcal{T}=\{T_\alpha : \alpha \in \Lambda\}$ and $\mathcal{U}=\{U_\alpha: \alpha \in \Lambda\}$ be $q$-commuting families of contractions and unitaries, respectively, with $\|q\|=1$ and both having the same $q$-intertwining relations. Clearly, $(U_{\alpha_1}, \dotsc, U_{\alpha_k})$ is a doubly $q_{\alpha}$-commuting tuple for $q_{\alpha}=\{q_{\alpha_i\alpha_j} : 1 \leq i < j \leq k \}$.  From (\ref{eqn402}), we have that 
	\[
	U_{\alpha_{1}}^{m_{\alpha_{1}}}\dotsc U_{\alpha_{k}}^{m_{\alpha_{k}}}=\underset{1\leq i<j \leq k}\prod q_{\alpha_{i}\alpha_{j}}^{-m_{\alpha_{i}}^+m_{\alpha_{j}}^-}\left[(U_{\alpha_{1}}^{m_{\alpha_{1}}^-})^*\dotsc (U_{\alpha_{k}}^{m_{\alpha_{k}}^-})^*\right]\left[U_{\alpha_{1}}^{m_{\alpha_{1}}^+}\dotsc U_{\alpha_{k}}^{m_{\alpha_{k}}^+}\right].
	\]
	Therefore, (\ref{eqn**}) is equivalent to the following:
	\[
	\left[(T_{\alpha_{1}}^{m_{\alpha_{1}}^-})^*\dotsc (T_{\alpha_{k}}^{m_{\alpha_{k}}^-})^*\right]\left[T_{\alpha_{1}}^{m_{\alpha_{1}}^+}\dotsc T_{\alpha_{k}}^{m_{\alpha_{k}}^+}\right]=P_\mathcal{H}\left[(U_{\alpha_{1}}^{m_{\alpha_{1}}^-})^*\dotsc (U_{\alpha_{k}}^{m_{\alpha_{k}}^-})^*\right]\left[U_{\alpha_{1}}^{m_{\alpha_{1}}^+}\dotsc U_{\alpha_{k}}^{m_{\alpha_{k}}^+}\right]\bigg|_\mathcal{H}.
	\]
	Consequently, for any permutation $\sigma$ on $\{\alpha_1, \dotsc, \alpha_k\}$, we have   
	\begin{equation}\label{eqn923}
		\underset{1\leq i<j \leq k}\prod q_{\beta_{i}\beta_{j}}^{-m_{\beta_{i}}^+m_{\beta_{j}}^-}\left[(T_{\beta_{1}}^{m_{\beta_{1}}^-})^*\dotsc (T_{\beta_{k}}^{m_{\beta_{k}}^-})^*\right]\left[T_{\beta_{1}}^{m_{\beta_{1}}^+}\dotsc T_{\beta_{k}}^{m_{\beta_{k}}^+}\right]=P_\mathcal{H}U_{\beta_{1}}^{m_{\beta_{1}}}\dotsc U_{\beta_{k}}^{m_{\beta_{k}}}|_\mathcal{H},
	\end{equation}
	where $\beta_j=\sigma(\alpha_j)$ for $1 \leq j \leq k$. Indeed, \eqref{eqn923} holds precisely because both $(T_{\alpha_1}, \dotsc, T_{\alpha_k})$ and $(U_{\alpha_1}, \dotsc, U_{\alpha_k})$ follow the same $q$-commutativity relations among themselves.	
	
	\smallskip
	
	From here onwards, we always assume some well-ordering on an indexing set $\Lambda$. Furthermore, if $\Lambda$ is a subset of the natural numbers, we always consider the natural order on the set $\Lambda$. Suppose $(T_1, \dotsc, T_k)$ is a $q$-commuting tuple of contractions acting on a Hilbert space $\mathcal{H}$ such that $q=\{q_{ij} \in \T \ : \  1 \leq i < j \leq k \}$. Let $(s_1, \dotsc, s_k)$ be a system of indeterminates corresponding to $(T_1, \dotsc, T_k).$ We need another system of indeterminates corresponding to each $q_{ij}$ and we shall denote them by $q_{ij}$ as well to avoid unnecessary symbols. Let $G_{dc}$ be the collection of the elements of the form 
	\begin{equation}\label{eqn8.03}
		\underset{1\leq i < j \leq k}{\prod}q_{ij}^{m_{ij}}s_1^{m_1}\dotsc s_k^{m_k} \quad (m_{ij}, m_1,\dotsc, m_k \in \Z)
	\end{equation}
	subject to the conditions 
	\begin{align}\label{eqn8.04}
		s_is_j=\left\{
		\begin{array}{ll}
			q_{ij}s_js_i, & i<j\\
			q_{ji}^{-1}s_js_i, & i>j 
		\end{array} 
		\right.,  
		\ \ 
		s_is_j^{-1}=\left\{
		\begin{array}{ll}
			q_{ij}^{-1}s_j^{-1}s_i, & i<j\\
			q_{ji}s_j^{-1}s_i, & i>j 
		\end{array} 
		\right., \ \  q_{ij}q_{mn}=q_{mn}q_{ij} \ \ \ \text{and} \ \ \  s_iq_{mn}=q_{mn}s_i,
	\end{align}
	for $1 \leq i< j \leq k$ and $1 \leq m<n \leq k$. The product in (\ref{eqn8.03}) is well-defined due to (\ref{eqn8.04}). For the ease of our computations, we shall denote by
	\[
	x^m=\underset{1\leq i < j \leq k }{\prod}q_{ij}^{m_{ij}}s_1^{m_1}\dotsc s_k^{m_k} \quad (m_{ij}, m_1,\dotsc, m_k \in \Z).
	\] 
	The order of the indeterminates matters.  We prove that the operations in (\ref{eqn8.04}) make $G_{dc}$ a group.
	
	\begin{lem}\label{group}
		Let $k \in \N$. The set 
		\begin{equation*}
			\begin{split}
				G_{dc}=\left\{\underset{1\leq i < j \leq k}{\prod}q_{ij}^{m_{ij}}s_1^{m_1}\dotsc s_k^{m_k} \ \bigg| \ m_{ij}, m_l \in \mathbb{Z}, \ 1 \leq i<j \leq k, \ 1 \leq l \leq k  \right \}
			\end{split}
		\end{equation*}
		forms a group under the multiplication operation as in $($\ref{eqn8.04}$)$.
	\end{lem}
	
	\begin{proof} By the principle of mathematical induction, it follows from (\ref{eqn8.04}) that  
		\[
		s_i^{\alpha}s_j^\beta=\left\{
		\begin{array}{ll}
			q_{ij}^{\alpha\beta}s_j^\beta s_i^\alpha, & i<j\\ \\
			q_{ji}^{-\alpha\beta}s_j^\beta s_i^\alpha, & i>j 
		\end{array} 
		\right. \qquad  (\alpha, \beta \in \Z).
		\]
		Let $x^m, x^n, x^r \in G_{dc}$. Then $x^mx^n  \in G_{dc}$ because,
		\begin{equation*}
			\begin{split}
				x^mx^n&=\left(\underset{1\leq i<j \leq k}{\prod}q_{ij}^{m_{ij}}s_1^{m_1}s_2^{m_2}\dotsc s_k^{m_k}\right) \left( \underset{1\leq i<j \leq k}{\prod}q_{ij}^{n_{ij}}s_1^{n_1}s_2^{n_2}\dotsc s_k^{n_k}\right)\\
				&=\underset{1\leq i<j \leq k}{\prod}q_{ij}^{m_{ij}+n_{ij}}\underset{1< j \leq k}{\prod}q_{1j}^{-n_1m_{j}} s_1^{m_1+n_1}s_2^{m_2}\dotsc s_k^{m_k}s_2^{n_2}\dotsc s_k^{n_k}\\
				&=\underset{1\leq i<j \leq k}{\prod}q_{ij}^{m_{ij}+n_{ij}}\underset{1< j \leq k}{\prod}q_{1j}^{-n_1m_{j}} \underset{2< j \leq k}{\prod}q_{2j}^{-n_2m_{j}} s_1^{m_1+n_1}s_2^{m_2+n_2}s_3^{m_3}\dotsc s_k^{m_k}s_3^{n_3}\dotsc s_k^{n_k}\\
				&\vdots\\
				&=\underset{1\leq i<j \leq k}{\prod}q_{ij}^{m_{ij}+n_{ij}}\underset{1 \leq i< j \leq k}{\prod}q_{ij}^{-n_im_{j}} s_1^{m_1+n_1}s_2^{m_2+n_2}\dotsc s_k^{m_k+n_k}\\
				&=\underset{1 \leq i< j \leq k}{\prod}q_{ij}^{-n_im_{j}}x^{m+n}.\\
			\end{split}
		\end{equation*}
		Furthermore, we have 
		\begin{equation*}
			\begin{split}
				(x^mx^n)x^r & =\underset{1 \leq i< j \leq k}{\prod}q_{ij}^{-n_im_{j}}x^{m+n}x^r=\underset{1 \leq i< j \leq k}{\prod}q_{ij}^{-n_im_{j}}\underset{1 \leq i< j \leq k}{\prod}q_{ij}^{-r_i(m_j+n_j)}x^{m+n+r}  		
			\end{split}
		\end{equation*}
		and 
		\begin{equation*}
			\begin{split}
				x^m(x^nx^r) & =x^m\underset{1 \leq i< j \leq k}{\prod}q_{ij}^{-r_in_{j}}x^{n+r}=\underset{1 \leq i< j \leq k}{\prod}q_{ij}^{-(n_i+r_i)m_j}\underset{1 \leq i< j \leq k}{\prod}q_{ij}^{-r_in_{j}}x^{m+n+r}. 		
			\end{split}
		\end{equation*}
		Thus $(x^mx^n)x^r=x^m(x^nx^r)$. The element $e=x^0$ serves as the identity of $G_{dc}$. For any $x^m \in G_{dc}$, it is not difficult to see that the inverse is given by 
		\[
				(x^m)^{-1}
				=\underset{1 \leq i< j \leq k}{\prod}q_{ij}^{-m_im_j}x^{-m},
				\]
		which is again an element of $G_{dc}$. The proof is complete. 
	\end{proof}
	
The classical theory of regular unitary dilation for commuting contractions is based on the notion of positive definite functions on a group and their	unitary representations which are defined below.
\begin{defn}
		Let $G$ be a group.
		\begin{enumerate}
			\item A function $T(s)$ on $G$, whose values are operators on a Hilbert space $\mathcal{H}$ is said to be \textit{positive definite} if $T\left(s^{-1}\right)=T(s)^* \text { for every } s \in G$ and the sum
			\begin{equation}\label{+}  
				\sum_{s \in G} \sum_{t \in G} \langle T(t^{-1} s) h(s), h(t)\rangle \geq 0, 
			\end{equation}
			for every $h \in c_{00}(G, \mathcal{H})$. Here $c_{00}(G, \mathcal{H})$ consists of maps $h: G \to \HS$ with finite support, i.e., $h$ has values different from zero on a finite subset of $G$ only.
			
			\item By a \textit{unitary representation} of the group $G$, we mean a function $U(s)$ on $G$, whose values are unitary operators on a Hilbert space $\mathcal{K}$ such that $U(e)=I$ ($e$ being the identity element of $G$) and $U(s) U(t)=U(st)$ for $s, t \in G$.
		\end{enumerate} 
	\end{defn}
Indeed, every positive definite function on a group possesses a unitary representation as the following celebrated theorem due to Naimark shows.
\begin{thm}[Naimark, \cite{NaimarkI}]\label{NaimarkI}
		For every positive definite function $T(s)$ on a group $G$ with identity $e$, whose values are operators on a Hilbert space $\mathcal{H}$ with $T(e)=I$, there is a
		unitary representation $U(s)$ of $G$ on a space $\mathcal{K}$ containing $\mathcal{H}$ such that
		\begin{equation*}
			T(s)=P_\mathcal{H}U(s)|_\mathcal{H} \quad(s \in G)  \quad \text{and} \quad 			\mathcal{K}=\bigvee_{s \in G}U(s)\mathcal{H} \quad(\mbox{minimality condition}).  
		\end{equation*}
		This unitary representation of $G$ is determined by the function $T(s)$ up to isomorphism. Conversely, given  a
		unitary representation $U(s)$ of $G$ on a space $\mathcal{K}$ and a subspace $\mathcal{H}$ of $\mathcal{K}$, the map $T: G \to \mathcal{B}(\mathcal{H})$ defined by 
		$
		T(s)=P_\mathcal{H}U(s)|_\mathcal{H}
		$
		is a positive definite function with $T(e)=I$.
	\end{thm}
Next comes the seminal theorem due to Brehmer \cite{Brehmer} (see Theorem \ref{thm_reg_Bre}) showing the equivalence of regular unitary dilation and Brehmer's positivity. A key step in the proof of Theorem \ref{thm_reg_Bre} for a commuting family $\{T_\alpha : \alpha \in \Lambda\}$ of contractions acting on a Hilbert space $\HS$ is to show that the map given by 
	\[
	T: \Z^\Lambda \to \mathcal{B}(\HS), \quad m \mapsto (T^{m^-})^*T^{m^+}
	\]
	is positive definite. Then an application of Naimark's Theorem gives the required family of commuting unitaries. To establish an analogue of this result in $q$-commuting setting, we need to overcome the following challenges. 
		\begin{enumerate}
		\item[(i)] Given a $q$-commuting family of contractions with $\|q\|=1$, we first need to construct an appropriate group that could replace $\Z^\Lambda$ suitably. 
		
		\item[(ii)] Once the group is constructed, we need to define a nice enough operator-valued function which we must show is positive definite.   
		
		\item[(iii)] The scalars in $q$ are preserved, i.e., for a $q$-commuting family $\mathcal{T}$ of contractions with $\|q\|=1$ on a Hilbert space $\HS$, we wish to have a family of unitaries $\mathcal{U}$ on a space $\mathcal{K} \supseteq \HS$ that dilates $\mathcal{T}$ and follows the same $q$-intertwining relations as $\mathcal{T}$. 
	\end{enumerate}
	Thus, our goal here is to apply Naimark's theorem and for this we construct a group $G_{dc}$ corresponding to a finite family $\underline{T}=(T_1, \dots , T_k)$ of $q$-commuting contractions with $\|q\|=1$ and then determine an appropriate operator-valued function $T$ on $G_{dc}$. After this, we follow the corresponding arguments from CH-I of \cite{Nagy}.

\smallskip
 
Unless mentioned otherwise, $G_{dc}$ denotes the group as in Lemma \ref{group}. Recall that for a contraction $C$  and $m \in \Z$, we define
	$
	C(m):=C^m$ if $m \geq 0$ and $C(m):=C^{*|m|}$ if $m<0$. For the sake of brevity, we often write
	\[
	G_{dc}=\left\{x^m=q^{m_0}s_1^{m_1}\dotsc s_k^{m_k} \bigg| m_{ij}, m_l \in \mathbb{Z}, \ 1 \leq i<j \leq k, \ 1 \leq l \leq k  \right \},\quad \text{where} \quad q^{m_0}=\underset{1\leq i < j \leq k}{\prod}q_{ij}^{m_{ij}}.	
		\]
Note that $q^{m_0}$ is used to denote an element in $G_{dc}$ as well as a complex scalar when multiplied with operators. From here onwards, we adopt the following way to represent $G_{dc}$:
\begin{equation}\label{eqn_Omega}
	G_{dc}=\{x^m : m \in \mathbb{Z}^{\Omega}\},
\end{equation}
where $\Omega$ is the indexing set $\{ij, l| 1\leq i<j \leq k, 1 \leq l \leq k\}$ and $\mathbb{Z}^{\Omega}$ is the set of the elements 
\[
m=(m_{12}, \dotsc, m_{1k}, m_{23}, \dotsc, m_{2k}, \dotsc, m_{k-1,k}, m_1, \dotsc , m_k)=(m_{ij} \vdots m_1, m_2, \dotsc, m_k)_{1 \leq i<j \leq k},
\]
with each entry in $\mathbb{Z}$. If all $m_{ij}, m_l \geq 0$, we write $m \geq 0$ and so, $n \geq m$ means that $n-m \geq 0.$ For arbitrary $n,m \in \mathbb{Z}^\Omega$, we set 
\begin{align*}
n \cup m & =(\max\{n_{ij}, m_{ij}\}\vdots, \max\{n_1, m_1\}, \dotsc, \max\{n_k, m_k\})_{1 \leq i<j \leq k}\\
n \cap m & =(\min\{n_{ij}, m_{ij}\}\vdots, \min\{n_1, m_1\}, \dotsc, \min\{n_k, m_k\})_{1 \leq i<j \leq k},
\end{align*}
which are again in $\mathbb{Z}^\Omega.$ Finally, we set $m^+=m\cup 0, m^-=-(m \cap 0)$ and so, $m^+-m^-=m$ for every $m \in \mathbb{Z}^\Omega$. We can retrieve the function $h$ from the function $g$. For $v \subset \Omega$, let us set $e(v)=(e_{ij}(v)\vdots e_1(v) \dotsc e_k(v))_{1 \leq i<j \leq k}$ in $\mathbb{Z}^\Omega$ such that
$
e_\omega(v)=\left\{
\begin{array}{ll}
	1, & \omega \in v\\
	0, & \omega \notin v
\end{array} 
\right. 
$. 

We will have our desired setting in the following theorem, which plays a crucial role in the proof of the $(2)\Leftrightarrow (3)$ parts of Theorem \ref{thm912}. 
 
	\begin{thm}\label{DCIII}
		Let $\underline T=(T_1, \dotsc, T_k)$ be a $q$-commuting tuple of contractions with $\|q\|=1$ on a Hilbert space $\mathcal{H}$. The map $T: G_{dc} \to \mathcal{B}(\mathcal{H})$ defined by 
		\[
		T\bigg(\underset{1\leq i<j \leq k}{\prod}q_{ij}^{m_{ij}}s_1^{m_1}\dotsc s_k^{m_k}\bigg)=\underset{1\leq i<j \leq k}{\prod}q_{ij}^{m_{ij}}\underset{1\leq i<j \leq k}\prod q_{ij}^{-m_{i}^+m_{j}^-}\left[(T_{1}^{m_1^-})^*\dotsc (T_{k}^{m_k^-})^*\right]\left[T_{1}^{m_1^+}\dotsc T_{k}^{m_k^+}\right]
		\]
		is positive definite if and only if $\underline T$ satisfies the Brehmer's positivity condition, i.e.,
		\begin{equation}\label{Brehmer's}
			S(u)=\sum_{v \subset u }(-1)^{|v|}T(x^{e(v)})^*T(x^{e(v)}) \geq 0
		\end{equation}
		for every subset $u$ of $\{1, \dotsc, k\}.$ 
	\end{thm}
 
\begin{proof} 
 
  Define an operator-valued function $T: G_{dc} \to \mathcal{B}(\mathcal{H})$ by
	\begin{equation}\label{Map2}
		T(q^{m_0}s_{1}^{m_1}\dotsc s_{k}^{m_k}):=q^{m_0}\underset{1\leq i<j \leq k}\prod q_{ij}^{-m_{i}^+m_{j}^-}\left[(T_{1}^{m_1^-})^*\dotsc (T_{k}^{m_k^-})^*\right]\left[T_{1}^{m_1^+}\dotsc T_{k}^{m_k^+}\right].
	\end{equation}
	We show that the map $T$ is positive definite. Note that $T(e)=I$. Take any $x^m \in G_{dc}$. Then
	\begin{equation*}
		\begin{split}
			& \ \quad T((x^m)^{-1})\\
			&=	T\bigg(\underset{1 \leq i< j \leq k}{\prod}q_{ij}^{-m_im_j}x^{-m}\bigg)\\
			&=\underset{1 \leq i< j \leq k}{\prod}q_{ij}^{-m_im_j}T(q^{-m_0}s_1^{-m_1}\dotsc s_k^{-m_k})=\underset{1 \leq i< j \leq k}{\prod}q_{ij}^{-m_im_j}T(q^{n_0}s_1^{n_1}\dotsc s_k^{n_k}) \quad (n_i=-m_i)\\
			&=q^{n_0}\underset{1 \leq i< j \leq k}{\prod}q_{ij}^{-m_im_j}\underset{1\leq i<j \leq k}\prod q_{ij}^{-n_{i}^+n_{j}^-}\left[(T_{1}^{n_1^-})^*\dotsc (T_{k}^{n_k^-})^*\right]\left[T_{1}^{n_1^+}\dotsc T_{k}^{n_k^+}\right] \\
			&=q^{-m_0}\underset{1 \leq i< j \leq k}{\prod}q_{ij}^{-m_im_j}\underset{1\leq i<j \leq k}\prod q_{ij}^{-m_{i}^-m_{j}^+}\left[(T_{1}^{m_1^+})^*\dotsc (T_{k}^{m_k^+})^*\right]\left[T_{1}^{m_1^-}\dotsc T_{k}^{m_k^-}\right] \\
			&=q^{-m_0}\underset{1 \leq i< j \leq k}{\prod}q_{ij}^{-m_im_j-m_{i}^-m_{j}^+}
			\left[\underset{1\leq i<j \leq k}\prod q_{ij}^{m_{i}^+m_{j}^+}(T_{k}^{m_k^+})^*\dotsc (T_{1}^{m_1^+})^*\right]
			\left[\underset{1\leq i<j \leq k}\prod q_{ij}^{m_{i}^-m_{j}^-}T_{k}^{m_k^-}\dotsc T_{1}^{m_1^-}\right] \\
			&=q^{-m_0}\underset{1 \leq i< j \leq k}{\prod}q_{ij}^{-m_im_j-m_{i}^-m_{j}^+ +m_{i}^+m_{j}^+ +m_{i}^-m_{j}^- }
			\left[(T_{k}^{m_k^+})^*\dotsc (T_{1}^{m_1^+})^*\right]
			\left[T_{k}^{m_k^-}\dotsc T_{1}^{m_1^-}\right] \\
			&=q^{-m_0}\underset{1 \leq i< j \leq k}{\prod}q_{ij}^{m_{i}^+m_{j}^-}
			\left[(T_{k}^{m_k^+})^*\dotsc (T_{1}^{m_1^+})^*\right]
			\left[T_{k}^{m_k^-}\dotsc T_{1}^{m_1^-}\right] \\
			&=T(x^m)^*.
		\end{split}
	\end{equation*}
As per the definition, the map $T$ in (\ref{Map2}) is positive definite if and only if
\[
\sum_{m, n \in \mathbb{Z}^\Omega} \left\langle T((x^n)^{-1}x^m)h(x^m),h(x^n)\right\rangle \geq 0 \quad \text{for all} \ h \in c_{00}(G_{dc}, \mathcal{H}).
\]
Let $h \in c_{00}(G_{dc}, \mathcal{H})$. The support of $h$ can be written as $\text{supp}(h)=\{x^m \in G_{dc} \ : \ -l \leq m \leq l\}$ for some $l \in \mathbb{Z}^\Omega.$  Define another function $\phi$ in $c_{00}(G_{dc}, \mathcal{H})$ as follows:
\[
\phi(x^m)=\underset{1 \leq i<j \leq k}{\prod}q_{ij}^{-m_il_j}h(x^{m-l}) \quad \mbox{which gives that} \quad h(x^m)=\underset{1 \leq i<j \leq k}{\prod}q_{ij}^{(m_i+l_i)l_j}\phi(x^{m+l}).
\]
Then
\begin{equation*}
	\begin{split}
		&\ \ \quad \sum_{m, n \in \mathbb{Z}^\Omega} \left\langle T((x^n)^{-1}x^m)h(x^m),h(x^n)\right\rangle\\
		&=\sum_{ -l \leq m, n \leq l} \left\langle T((x^n)^{-1}x^m)h(x^m),h(x^n)\right\rangle\\
		&=		\sum_{ -l \leq m, n \leq l} \left\langle T((x^n)^{-1}x^m)\underset{1 \leq i<j \leq k}{\prod}q_{ij}^{(m_i+l_i)l_j}\phi(x^{m+l}),\underset{1 \leq i<j \leq k}{\prod}q_{ij}^{(n_i+l_i)l_j}\phi(x^{n+l})\right\rangle\\
		&=		\sum_{ m, n \geq 0} \left\langle T((x^{n-l})^{-1}x^{m-l})\underset{1 \leq i<j \leq k}{\prod}q_{ij}^{m_il_j}\phi(x^{m}),\underset{1 \leq i<j \leq k}{\prod}q_{ij}^{n_il_j}\phi(x^{n})\right\rangle.\\
	\end{split}
\end{equation*}
Using the fact that 
\[
x^nx^m=\underset{1 \leq i< j \leq k}{\prod}q_{ij}^{-m_in_{j}}x^{n+m} \quad  \mbox{and} \quad  (x^m)^{-1}=\underset{1 \leq i< j \leq k}{\prod}q_{ij}^{-m_im_{j}}x^{-m},
\]
we have that 
\begin{equation*}
	\begin{split}
		T((x^{n-l})^{-1}x^{m-l})
		&=T\left(\underset{1 \leq i< j \leq k}{\prod}q_{ij}^{-(n_i-l_i)(n_j-l_j)}x^{-n+l}x^{m-l}\right)\\
		&=T\left(\underset{1 \leq i< j \leq k}{\prod}q_{ij}^{-(n_i-l_i)(n_j-l_j)}\underset{1 \leq i< j \leq k}{\prod}q_{ij}^{-(m_i-l_i)(-n_j+l_j)}x^{m-n}\right)\\
		&=T\left(\underset{1 \leq i< j \leq k}{\prod}q_{ij}^{-(m_i-n_i)(l_j-n_j)}x^{m-n}\right)\\
		&=\underset{1 \leq i< j \leq k}{\prod}q_{ij}^{-(m_i-n_i)(l_j-n_j)}T(x^{m-n}).\\
	\end{split}
\end{equation*}

Putting everything together, we have that

\begin{equation*}
	\begin{split}
		\sum_{m, n \in \mathbb{Z}^\Omega} \left\langle T((x^n)^{-1}x^m)h(x^m),h(x^n)\right\rangle 
		&=\sum_{ m, n \geq 0} \left\langle T((x^{n-l})^{-1}x^{m-l})\underset{1 \leq i<j \leq k}{\prod}q_{ij}^{(m_i-n_i)l_j}\phi(x^{m}),\phi(x^{n})\right\rangle\\
		&=		\sum_{ m, n \geq 0} \left\langle \underset{1 \leq i< j \leq k}{\prod}q_{ij}^{(m_i-n_i)n_j}T(x^{m-n})\phi(x^{m}),\phi(x^{n})\right\rangle.\\
	\end{split}
\end{equation*}
Consequently, the map $T$ as in (\ref{Map2}) is positive definite if and only if 

\begin{equation}\label{eqn802}
	\sum_{ m, n \geq 0} \left\langle \underset{1 \leq i< j \leq k}{\prod}q_{ij}^{(m_i-n_i)n_j}T(x^{m-n})h(x^{m}),h(x^{n})\right\rangle \geq 0
\end{equation}
for every $h \in c_{00}(G_{dc}, \mathcal{H}).$ Now we make the reciprocal formulas. Let us first observe that for every function $h \in c_{00}(G_{dc}, \mathcal{H}),$ the function
\begin{equation}\label{eqn803}
	g(x^n)=	\sum_{ m \geq  n}\left( \underset{1 \leq i< j \leq k}{\prod}q_{ij}^{(m_i-n_i)n_j}T(x^{m-n})h(x^{m})\right)  \quad (n \geq 0 \ \mbox{in} \ \mathbb{Z}^{\Omega})\\
\end{equation}
is also in $c_{00}(G_{dc}, \mathcal{H})$. The reciprocal formula for (\ref{eqn803}) is given by
\begin{equation}\label{eqn804}
	h(x^n)=\sum_{v \subset \Omega}(-1)^{|v|}\underset{1 \leq i< j \leq k}{\prod}q_{ij}^{e_i(v)n_j}T(x^{e(v)})g(x^{n+e(v)})\quad (n \geq 0 \ \mbox{in} \ \mathbb{Z}^{\Omega}),
\end{equation}
where as indicated $v$ runs over all the subsets of $\Omega$. The map $g$ is finitely non-zero, and hence there are only finitely many non-zero terms in the sum (\ref{eqn804}). In fact, for any fixed $n\geq 0$ in $\mathbb{Z}^\Omega$, we have 
\begin{equation*}
	\begin{split}
		& \quad 	\sum_{v \subset \Omega}(-1)^{|v|}\underset{1 \leq i< j \leq k}{\prod}q_{ij}^{e_i(v)n_j}T(x^{e(v)})g(x^{n+e(v)})\\
		&=\sum_{v \subset \Omega}(-1)^{|v|}\underset{1 \leq i< j \leq k}{\prod}q_{ij}^{e_i(v)n_j}T(x^{e(v)})\bigg[\sum_{ m \geq  n+e(v)}\left( \underset{1 \leq i< j \leq k}{\prod}q_{ij}^{(m_i-n_i-e_i(v))(n_j+e_j(v))}T(x^{m-n-e(v)})h(x^{m})\right)\bigg]	
		\\
		&=\sum_{\substack{v \subset \Omega \\ m \geq  n+e(v)}}(-1)^{|v|}\underset{1 \leq i< j \leq k}{\prod}q_{ij}^{(m_i-n_i-e_i(v))e_j(v)+(m_i-n_i)n_j}\bigg[T(x^{e(v)})\left( T(x^{m-n-e(v)})\right)\bigg]h(x^{m})\\
		&=\sum_{\substack{v \subset \Omega \\ m \geq  n+e(v)}}(-1)^{|v|}\underset{1 \leq i< j \leq k}{\prod}q_{ij}^{(m_i-n_i-e_i(v))e_j(v)+(m_i-n_i)n_j}\bigg[\underset{1 \leq i< j \leq k}{\prod}q_{ij}^{-(m_i-n_i-e_i(v))e_j(v)}T(x^{m-n})\bigg]h(x^{m})\\
		&=\sum_{\substack{v \subset \Omega \\ m \geq  n+e(v)}}(-1)^{|v|}\underset{1 \leq i< j \leq k}{\prod}q_{ij}^{(m_i-n_i)n_j}T(x^{m-n})h(x^{m})\\	
		&=\sum_{m \geq n}\bigg[\sum_{e(v) \leq m-n}(-1)^{|v|}\bigg]\underset{1 \leq i< j \leq k}{\prod}q_{ij}^{(m_i-n_i)n_j}T(x^{m-n})h(x^{m})\\
		&=h(x^n),	
	\end{split}
\end{equation*}
where we have used the fact that 
\begin{equation}\label{eqn805.0}
	T(x^n)T(x^m)=\underset{1 \leq i< j \leq k}{\prod}q_{ij}^{-m_in_j}T(x^{n+m}) \quad (n, m \geq 0 \ \text{in} \ \Z^\Omega)
\end{equation}
and if $v$ runs through all subsets of a finite set $v_0$ (including the empty set and the whole set) then
\begin{equation}\label{eqn805}
	\sum_{v \subset v_0}(-1)^{|v|}=\left\{
	\begin{array}{ll}
		1 & \ \mbox{if} \ v_0 \ \mbox{is empty},\\
		0 & \ \mbox{if} \ v_0 \ \mbox{is not empty.}
	\end{array} 
	\right. 
\end{equation}
Conversely, if one starts with an arbitrary function $g(x^n) (n \geq 0 \ \mbox{in} \ \mathbb{Z}^\Omega)$ in $c_{00}(G_{dc}, \mathcal{H})$
then the function $h(x^n)$, which it generates by the formula (\ref{eqn804}), is also in $c_{00}(G_{dc}, \mathcal{H})$. For every fixed $n \geq 0$ in $\mathbb{Z}^\Omega,$ we have
\begin{equation*}
	\begin{split}
		\quad			& \sum_{ m \geq  n}\left[ \underset{1 \leq i< j \leq k}{\prod}q_{ij}^{(m_i-n_i)n_j}T(x^{m-n})h(x^{m})\right]\\
		&= \sum_{ m \geq  n}\left[ \underset{1 \leq i< j \leq k}{\prod}q_{ij}^{(m_i-n_i)n_j}T(x^{m-n})\bigg(\sum_{v \subset \Omega}(-1)^{|v|}\underset{1 \leq i< j \leq k}{\prod}q_{ij}^{e_i(v)m_j}T(x^{e(v)})g(x^{m+e(v)})\bigg)\right]\\
		&= \sum_{\substack{m \geq  n\\ v \subset \Omega}}(-1)^{|v|} \underset{1 \leq i< j \leq k}{\prod}q_{ij}^{(m_i-n_i)n_j+e_i(v)m_j}\bigg[T(x^{m-n})T(x^{e(v)})\bigg]g(x^{m+e(v)})\\
		&= \sum_{\substack{m \geq  n\\ v \subset \Omega}}(-1)^{|v|} \underset{1 \leq i< j \leq k}{\prod}q_{ij}^{(m_i-n_i)n_j+e_i(v)m_j}\bigg[\underset{1 \leq i< j \leq k}{\prod}q_{ij}^{-e_i(v)(m_j-n_j)}T(x^{m-n+e(v)})\bigg]g(x^{m+e(v)}) \quad [\text{by} (\ref{eqn805.0})]\\
		&= \sum_{v \subset \Omega} \sum_{m \geq  n}(-1)^{|v|} \underset{1 \leq i< j \leq k}{\prod}q_{ij}^{(m_i-n_i+e_i(v))n_j}T(x^{m-n+e(v)})g(x^{m+e(v)})\\
		&=\sum_{v \subset \Omega} (-1)^{|v|} \sum_{p \geq  n+e(v)} \underset{1 \leq i< j \leq k}{\prod}q_{ij}^{(p_i-n_i)n_j}T(x^{p-n})g(x^{p})\\
		&=\sum_{p \geq n}\bigg[\sum_{e(v) \leq p-n} (-1)^{|v|}\bigg]\underset{1 \leq i< j \leq k}{\prod}q_{ij}^{(p_i-n_i)n_j}T(x^{p-n})g(x^{p})\\
		&=g(x^n),
	\end{split}
\end{equation*}
again by virtue of (\ref{eqn805}). We conclude that the formulas (\ref{eqn803}) and (\ref{eqn804}) give a transformation of $c_{00}(G_{dc}, \mathcal{H})$ onto itself (defined for $n \geq 0$ in $\mathbb{Z}^\Omega$) and the inverse of this transformation respectively. Consequently, in order that (\ref{eqn802}) hold for every $h(x^n) (n \geq 0)$ in $c_{00}(G_{dc}, \mathcal{H})$, it is necessary and sufficient that the sum
\begin{small} 
	\begin{equation*}
		\sum_{ m, n \geq 0}\sum_{\substack {v \subset \Omega \\ w \subset \Omega}}(-1)^{|v|+|w|} \left\langle \underset{1 \leq i< j \leq k}{\prod}q_{ij}^{(m_i-n_i)n_j+e_i(v)m_j-e_i(w)n_j}T(x^{m-n})T(x^{e(v)})g(x^{m+e(v)}),T(x^{e(w)})g(x^{n+e(w)})\right\rangle
	\end{equation*}
\end{small} 
be non-negative for every $g(x^n) (n \geq 0).$ Now this sum can be re-written as 
\begin{equation}\label{eqn806}
	\sum_{p \geq 0} \sum_{r \geq 0} \langle D(p,r)g(x^p), g(x^r) \rangle, 
\end{equation}
where $D(p,r)$ equals 
\begin{equation*}
	\sum_{\substack{v \subset \pi(p) \\ w \subset \pi(r)}}(-1)^{|v|+|w|}\underset{1 \leq i< j \leq k}{\prod}q_{ij}^{(p_i-e_i(v)-r_i)(r_j-e_j(w))+(p_j-e_j(v))e_i(v)}T(x^{e(w)})^*T(x^{p-e(v)-r+e(w)})T(x^{e(v)})
\end{equation*}
and for $n$ in $\mathbb{Z}^\Omega,$ the set $\pi(n)$ is defined as
\[
\pi(n)=\{\omega \in \Omega: n_\omega >0 \}.
\]
We now prove that $D(p,r)=0$ if $p \ne r$. Observe that if $p \ne r$, then the sets $\pi(p-r)$ and $\pi(r-p)$ cannot both be empty. By reason of symmetry, it is sufficient to consider the case when $\pi(p-r)$ is not empty. The set
	\[
	\delta(w)=\pi(p) \cap \pi(p-r+e(w))
	\]
	is then non-empty for every subset $w$ of $\Omega$ since it contains the set $\pi(p-r).$ Next, we observe that for every fixed $w \subset \Omega,$ one obtains all subsets $v$ of $\pi(p)$ by taking $v=v' \cup v'',$ where
	\[
	v'=v \cap (\pi(p)\setminus\delta(w)) \quad  \mbox{and} \quad v''=v \cap \delta(w).
	\]
	We have $|v|=|v'|+|v''|$ and $e(v)=e(v')+e(v'').$ Then $D(p,r)$ equals
	\begin{small} 
		\begin{equation}\label{bracketI}
			\sum_{ w \subset \pi(r)}(-1)^{|w|}T(x^{e(w)})^*\sum_{v' \subset \pi(p)\setminus \delta(w)}(-1)^{|v'|}\bigg[\sum_{v'' \subset \delta(w)}(-1)^{|v''|}\underset{1 \leq i< j \leq k}{\prod}q_{ij}^{\xi_{ij}(p;r)}T(x^{p-r+e(w)-e(v)})T(x^{e(v)})\bigg],
		\end{equation}
	\end{small} 
	where $\xi_{ij}(p;r)=(p_i-e_i(v)-r_i)(r_j-e_j(w))+(p_j-e_j(v))e_i(v).$ We use the substitution 
	\[
	a=p-r+e(w)-e(v)=p-r+e(w)-e(v')-e(v'') \quad \mbox{and} \quad b=a^+ + e(v'').
	\]
	Note that $a^-$ and $b$ do not depend upon $v'';$ namely we have 
	\[
	a^-=[p-r+e(w)-e(v')]^- \quad \mbox{and} \quad b=[p-r+e(w)-e(v')]^+,
	\]
	which is a consequence of the fact that for $\omega \in v''$ we have $p_\omega -r_\omega +e_\omega(w) \geq 1$ and $e_\omega(v')=0.$ 
	Now we compute the following.
	\begin{align}\label{step2.1}
		T(x^a)T(x^{e(v)})
		&=T\left(\prod_{1\leq i < j \leq k}q_{ij}^{a_{ij}}s_1^{a_1}\dotsc s_k^{a_k}\right)T\left(\prod_{1 \leq i <j \leq k}q_{ij}^{e_{ij}(v)}s_1^{e_1(v)}\dotsc s_k^{e_k}\right) \notag \\
		&=	\prod_{1\leq i < j \leq k}q_{ij}^{a_{ij}+e_{ij}(v)-a_i^+a_j^-}\bigg[(T_k^{a_k^-}\dotsc T_1^{a_1^-})^* T_1^{a_1^+}\dotsc T_k^{a_k^+}\bigg]T_1^{e_1(v)}\dotsc T_k^{e_k(v)} \notag \\	
		&=	\prod_{1\leq i < j \leq k}q_{ij}^{a_{ij}+e_{ij}(v)-a_i^+a_j^-}(T_k^{a_k^-}\dotsc T_1^{a_1^-})^* \bigg[T_1^{a_1^+}\dotsc T_k^{a_k^+}T_1^{e_1(v)}\dotsc T_k^{e_k(v)}\bigg] \notag \\	
		&= 	\prod_{1\leq i < j \leq k}q_{ij}^{a_{ij}+e_{ij}(v)-a_i^+a_j^-}(T_k^{a_k^-}\dotsc T_1^{a_1^-})^*\bigg[T_1^{a_1^+}\dotsc T_k^{a_k^+}T_1^{e_1(v'')}T_1^{e_1(v')}\dotsc T_k^{e_k(v'')}T_k^{e_k(v')}\bigg].
	\end{align}
	The bracket in (\ref{step2.1}) can be written as
	\begin{align}\label{step2.2}
		T_1^{a_1^+}\dotsc T_k^{a_k^+}\bigg[T_1^{e_1(v'')}T_1^{e_1(v')}\dotsc T_k^{e_k(v'')}T_k^{e_k(v')}\bigg]
		&=\prod_{1 \leq i <j \leq k}q_{ij}^{e_i(v')e_j(v'')-e_i(v'')a_j^+} T_1^{b_1}\dotsc T_k^{b_k}T_1^{e_1(v')}\dotsc T_k^{e_k(v')}. \notag \\
	\end{align}
	Hence, $T(x^a)T(x^{e(v)})$ equals
	\begin{equation*}
		\begin{split} 
			&	\prod_{1\leq i < j \leq k}q_{ij}^{a_{ij}+e_{ij}(v)-a_i^+a_j^-}(T_k^{a_k^-}\dotsc T_1^{a_1^-})^*\bigg[\prod_{1 \leq i <j \leq k}q_{ij}^{e_i(v')e_j(v'')-e_i(v'')a_j^+} T_1^{b_1}\dotsc T_k^{b_k}T_1^{e_1(v')}\dotsc T_k^{e_k(v')}\bigg]\\
			&=\prod_{1\leq i < j \leq k}q_{ij}^{a_{ij}+e_{ij}(v)-a_i^+a_j^-+e_i(v')e_j(v'')-e_i(v'')a_j^+}(T_k^{a_k^-}\dotsc T_1^{a_1^-})^*\bigg[ T_1^{b_1}\dotsc T_k^{b_k}T_1^{e_1(v')}\dotsc T_k^{e_k(v')}\bigg]\\
			&=\prod_{1\leq i < j \leq k}q_{ij}^{a_{ij}+e_{ij}(v)-a_i^+a_j^-+e_i(v')e_j(v'')-e_i(v'')a_j^++a_i^-a_j^-}(T_1^{a_1^-}\dotsc T_k^{a_k^-})^*\bigg[ T_1^{b_1}\dotsc T_k^{b_k}T_1^{e_1(v')}\dotsc T_k^{e_k(v')}\bigg]\\
			&=\prod_{1\leq i < j \leq k}q_{ij}^{\zeta_{ij}(v)}T(x^{a^-})^*T(x^b)T(x^{e(v')}),\\
		\end{split}
	\end{equation*}
	where
	\begin{equation*}
		\begin{split} 
			\zeta_{ij}(v)=-a_i^+a_j^-+e_i(v')e_j(v'')-e_i(v'')a_j^++a_i^-a_j^-
			&=-(a_i^+-a_i^-)a_j^-+e_i(v')e_j(v'')-e_i(v'')a_j^+\\
			&=-a_ia_j^-+e_i(v')e_j(v'')-e_i(v'')a_j^+.
		\end{split} 
	\end{equation*}
	Therefore, the bracket in (\ref{bracketI}) can be simplified as  
	\begin{small}
		\begin{equation}\label{step5}
			\sum_{v'' \subset \delta(w)}(-1)^{|v''|}\underset{1 \leq i< j \leq k}{\prod}q_{ij}^{\xi_{ij}(p;r)}T(x^{a})T(x^{e(v)})=\sum_{v'' \subset \delta(w)}(-1)^{|v''|}\underset{1 \leq i< j \leq k}{\prod}q_{ij}^{\xi_{ij}(p;r)+\zeta_{ij}(v)}T(x^{a^-})^*T(x^b)T(x^{e(v')}).
		\end{equation}
	\end{small}
	We compute the exponent $\xi_{ij}(p;r)+\zeta_{ij}(v)$ appearing in (\ref{step5}) which equals
	\begin{align}\label{step6}
		&=	(p_i-e_i(v)-r_i)(r_j-e_j(w))+(p_j-e_j(v))e_i(v)-a_ia_j^-+e_i(v')e_j(v'')-e_i(v'')a_j^+ \notag  \\
		&\begin{aligned}
			=&	(p_i-e_i(v')-e_i(v'')-r_i)(r_j-e_j(w))+(p_j-e_j(v')-e_j(v''))(e_i(v')+e_i(v''))\notag  \\
			& - (p_i-r_i+e_i(w)-e_i(v')-e_i(v''))a_j^-+e_i(v')e_j(v'')-e_i(v'')a_j^+
		\end{aligned}
		\notag 	\\
		&\begin{aligned}
			=&	(p_i-r_i-e_i(v'))(r_j-e_j(w))-(p_i-r_i+e_i(w)-e_i(v'))a_j^-
			\notag 	\\
			& +e_i(v'')\bigg[ (p_j-r_j+e_j(w)-e_j(v')-e_j(v''))-a_j\bigg] +(p_j-e_j(v'))e_i(v')
		\end{aligned}
		\notag  \\
		&=	(p_i-r_i-e_i(v'))(r_j-e_j(w))-(p_i-r_i+e_i(w)-e_i(v'))a_j^- +(p_j-e_j(v'))e_i(v').
	\end{align}
	This shows that the exponents $\xi_{ij}(p;r)+\zeta_{ij}(v)$ do not depend upon $v''.$ Consequently, using (\ref{step5}) and (\ref{step6}), the bracket in (\ref{bracketI}) is given by the following.
	\begin{align}\label{step7}
		& \quad \ \sum_{v'' \subset \delta(w)}(-1)^{|v''|}\underset{1 \leq i< j \leq k}{\prod}q_{ij}^{\xi_{ij}(p;r)}T(x^{a})T(x^{e(v)}) \notag \\
		& =\sum_{v'' \subset \delta(w)}(-1)^{|v''|}\underset{1 \leq i< j \leq k}{\prod}q_{ij}^{\xi_{ij}(p;r)+\zeta_{ij}(v)}T(x^{a^-})^*T(x^b)T(x^{e(v')}) \notag  \\
		&=\underset{1 \leq i< j \leq k}{\prod}q_{ij}^{\xi_{ij}(p;r)+\eta_{ij}(v)}T(x^{a^-})^*T(x^b)\bigg[\sum_{v'' \subset \delta(w)}(-1)^{|v''|}\bigg]T(x^{e(v')}) \notag \\
		&=0.	
	\end{align}
	The last equality holds because the set $\delta(w)$ is non-empty. Substituting (\ref{step7}) back in (\ref{bracketI}), we have that $D(p,r)=0$ for every $p \ne r.$ Therefore, the sum as in (\ref{eqn806}) reduces to
	\begin{equation*}
		\sum_{p \geq 0} \langle D(p,p)g(x^p), g(x^p) \rangle.	
	\end{equation*}
	In order that this sum be non-negative for every $g(x^p) (p \geq 0 \ \mbox{in} \ \mathbb{Z}^\Omega)$, it is necessary and sufficient that $D(p,p) \geq 0$ for every $p \geq 0$ in $\mathbb{Z}^\Omega$. Note that
	\begin{equation}\label{D(p,p)}
		\begin{split}
			D(p,p)&=	\sum_{\substack{v \subset \pi(p) \\ w \subset \pi(p)}}(-1)^{|v|+|w|}\underset{1 \leq i< j \leq k}{\prod}q_{ij}^{-e_i(v)(p_j-e_j(w))+(p_j-e_j(v))e_i(v)}T(x^{e(w)})^*T(x^{e(w)-e(v)})T(x^{e(v)})\\
			&=	\sum_{\substack{v \subset \pi(p) \\ w \subset \pi(p)}}(-1)^{|v|+|w|}\underset{1 \leq i< j \leq k}{\prod}q_{ij}^{e_i(v)(e_j(w)-e_j(v))}\bigg[T(x^{e(w)})^*T(x^{e(w)-e(v)})T(x^{e(v)})\bigg].\\
		\end{split}
	\end{equation}
	For fixed $v, w \subset \pi(p),$ we denote $c=e(w)-e(v).$ In (\ref{D(p,p)}), the term 
	\[
	\bigg[T(x^{e(w)})^*T(x^{e(w)-e(v)})T(x^{e(v)})\bigg]
	\]
	equals
	\[
	\bigg[(T_k^{e_k(w)})^*\dotsc (T_1^{e_1(w)})^*\bigg]\bigg[\underset{1 \leq i<j \leq k}{\prod}q_{ij}^{-c_i^+c_j^-}(T_1^{c_1^-})^*\dotsc (T_k^{c_k^-})^*T_1^{c_1^+}\dotsc T_k^{c_k^+}\bigg]\bigg[T_1^{e_1(v)}\dotsc T_k^{e_k(v)}\bigg]
	\] 
	\begin{align}\label{step8}
		&=\underset{1 \leq i<j \leq k}{\prod}q_{ij}^{-c_i^+c_j^-}\bigg[(T_k^{e_k(w)})^*\dotsc (T_1^{e_1(w)})^*(T_1^{c_1^-})^*\dotsc (T_k^{c_k^-})^*\bigg]\bigg[T_1^{c_1^+}\dotsc T_k^{c_k^+}T_1^{e_1(v)}\dotsc T_k^{e_k(v)}\bigg] \notag \\
		&=
		\underset{1 \leq i<j \leq k}{\prod}q_{ij}^{-c_i^+c_j^-}\bigg[\underset{1 \leq i<j \leq k}{\prod}q_{ij}^{(c_i^-+e_i(w))c_j^-}(T_k^{c_k^-+e_k(w)})^*\dotsc (T_1^{c_1^-+e_1(w)})^*\bigg] \notag \\
		& \qquad \ \ \ \quad \quad \qquad \ \bigg[\underset{1 \leq i<j \leq k}{\prod}q_{ij}^{-e_i(v)c_j^+}T_1^{c_1^++e_1(v)}\dotsc T_k^{c_k^++e_k(v)}\bigg]
		\notag  \\
		&=\underset{1 \leq i<j \leq k}{\prod}q_{ij}^{-c_i^+c_j^-+(c_i^-+e_i(w))c_j^--e_i(v)c_j^+}\bigg[(T_k^{c_k^-+e_k(w)})^*\dotsc (T_1^{c_1^-+e_1(w)})^*\bigg]\bigg[T_1^{c_1^++e_1(v)}\dotsc T_k^{c_k^++e_k(v)}\bigg]\notag  \\
		&=\underset{1 \leq i<j \leq k}{\prod}q_{ij}^{-c_i^+c_j^-+(c_i^-+e_i(w))c_j^--e_i(v)c_j^+}\bigg(\prod_{1 \leq i \leq k}T_i^{e_i(u)}\bigg)^*\bigg(\prod_{1 \leq i \leq k}T_i^{e_i(u)}\bigg) \quad (u=v \cup w).
	\end{align}
	We compute the following.
	\begin{equation*}
		\begin{split}
			\underset{1 \leq i< j \leq k}{\prod}q_{ij}^{e_i(v)c_j}\underset{1 \leq i<j \leq k}{\prod}q_{ij}^{-c_i^+c_j^-+(c_i^-+e_i(w))c_j^--e_i(v)c_j^+}&=
			\underset{1 \leq i< j \leq k}{\prod}q_{ij}^{e_i(v)c_j-c_i^+c_j^-+(c_i^-+e_i(w))c_j^--e_i(v)c_j^+}\\
			&=\underset{1 \leq i< j \leq k}{\prod}q_{ij}^{-e_i(v)c_j^--c_i^+c_j^-+(c_i^-+e_i(w))c_j^-}\\
			&=\underset{1 \leq i< j \leq k}{\prod}q_{ij}^{(e_i(w)-e_i(v))c_j^--c_i^+c_j^-+c_i^-c_j^-}\\
			&=\underset{1 \leq i< j \leq k}{\prod}q_{ij}^{c_ic_j^--c_i^+c_j^-+c_i^-c_j^-}\\
			&=1.
		\end{split} 
	\end{equation*}
	Putting the above computations in (\ref{D(p,p)}), we get that $D(p,p)$ is equal to
	\begin{small} 
		\begin{align}\label{D(p,p)II}
			&	\sum_{\substack{v \subset \pi(p) \\ w \subset \pi(p)}}(-1)^{|v|+|w|}\underset{1 \leq i< j \leq k}{\prod}q_{ij}^{e_i(v)c_j}\bigg[\underset{1 \leq i<j \leq k}{\prod}q_{ij}^{-c_i^+c_j^-+(c_i^-+e_i(w))c_j^--e_i(v)c_j^+}\bigg(\prod_{1 \leq i \leq k}T_i^{e_i(u)}\bigg)^*\bigg(\prod_{1 \leq i \leq k}T_i^{e_i(u)}\bigg)\bigg] \notag \\
			&=	\sum_{\substack{v \subset \pi(p) \\ w \subset \pi(p)}}(-1)^{|v|+|w|}\bigg(\prod_{1 \leq i \leq k}T_i^{e_i(u)}\bigg)^*\bigg(\prod_{1 \leq i \leq k}T_i^{e_i(u)}\bigg) \notag \\
			&=\sum_{\substack{v \subset \pi(p) \\ w \subset \pi(p)}}(-1)^{|v|+|w|}T(x^{e(u)})^*T(x^{e(u)})
			=\sum_{u \subset \pi(p) }(-1)^{|u|}T(x^{e(u)})^*T(x^{e(u)}).
		\end{align}
	\end{small} 
	The proof is now complete.
	\end{proof}
	
Below we prove the $(2)\Leftrightarrow (3)$ parts of Theorem \ref{thm912} for a finite $q$-commuting family which follow from Theorem \ref{NaimarkI} and Theorem \ref{DCIII}.

\begin{thm}\label{regular q-unitary dilation}
	Let $\underline{T}=(T_1, \dotsc, T_k)$ be a $q$-commuting tuple of contractions with $\|q\|=1$ acting on a Hilbert space $\HS$. Then $\underline{T}$ admits a regular $Q$-unitary dilation if and only if it satisfies the Brehmer's positivity condition, i.e.,
	\[
	S(u)=\sum_{v \subset u }(-1)^{|v|}T(x^{e(v)})^*T(x^{e(v)}) \geq 0
	\] 
	for every subset $u$ of $\{1, \dotsc, k\}$, where $T$ is the map as in Theorem \ref{DCIII}. 
\end{thm}

\begin{proof}
Let $\underline{U}=(U_1, \dotsc, U_k)$ acting on a Hilbert space $\mathcal{K}$ be a regular $Q$-unitary dilation of the $q$-commuting tuple $\underline{T}$.  We show that $\underline{T}$ satisfies the Brehmer's positivity condition. By Theorem \ref{DCIII}, it suffices to show that the map $T: G_{dc} \to \mathcal{B}(\HS)$ given by
\[
T\bigg(\underset{1\leq i<j \leq k}{\prod}q_{ij}^{m_{ij}}s_1^{m_1}\dotsc s_k^{m_k}\bigg)=\underset{1\leq i<j \leq k}{\prod}q_{ij}^{m_{ij}}\underset{1\leq i<j \leq k}\prod q_{ij}^{-m_{i}^+m_{j}^-}\left[(T_{1}^{m_1^-})^*\dotsc (T_{k}^{m_k^-})^*\right]\left[T_{1}^{m_1^+}\dotsc T_{k}^{m_k^+}\right]
\]
is positive definite. Let us consider the map $U: G_{dc} \to \mathcal{B}(\mathcal{K})$ given by
\[
U\bigg(\underset{1\leq i<j \leq k}{\prod}q_{ij}^{m_{ij}}s_1^{m_1}\dotsc s_k^{m_k}\bigg)=\underset{1\leq i<j \leq k}{\prod}Q_{ij}^{m_{ij}}\underset{1\leq i<j \leq k}\prod Q_{ij}^{-m_{i}^+m_{j}^-}\left[(U_{1}^{m_1^-})^*\dotsc (U_{k}^{m_k^-})^*\right]\left[U_{1}^{m_1^+}\dotsc U_{k}^{m_k^+}\right].
\]
Evidently, $U$ is an identity preserving map. Since $U_1, \dotsc, U_k$ are $Q$-commuting unitaries, it follows that $U_1, \dotsc, U_k$ are doubly $Q$-commuting. Following the similar computations as in \eqref{eqn402}, we have
\[
U\bigg(\underset{1\leq i<j \leq k}{\prod}q_{ij}^{m_{ij}}s_1^{m_1}\dotsc s_k^{m_k}\bigg)=\underset{1\leq i<j \leq k}{\prod}Q_{ij}^{m_{ij}}U_1^{m_1}\dotsc U_k^{m_k}.
\]
For $x^m=\underset{1\leq i<j \leq k}{\prod}q_{ij}^{m_{ij}}s_1^{m_1}\dotsc s_k^{m_k}$ and $x^n=\underset{1\leq i<j \leq k}{\prod}q_{ij}^{n_{ij}}s_1^{n_1}\dotsc s_k^{n_k}$ in $G_{dc}$, we have by Lemma \ref{group} that 
\[
x^mx^n=\underset{1\leq i<j \leq k}{\prod}q_{ij}^{m_{ij}+n_{ij}}\underset{1 \leq i< j \leq k}{\prod}q_{ij}^{-n_im_{j}} s_1^{m_1+n_1}s_2^{m_2+n_2}\dotsc s_k^{m_k+n_k}
=\underset{1 \leq i< j \leq k}{\prod}q_{ij}^{-n_im_{j}}x^{m+n}.
\]
Using the fact that $U_i^{m_i}U_j^{m_j}=Q_{ij}^{m_im_j}U_j^{m_j}U_i^{m_i}$ for all $m_i, m_j \in \Z$ and $1 \leq i, j \leq k$ with $i \ne j$, we have
\begin{equation*}
	\begin{split}
		U(x^mx^n)
		=U\left(\underset{1 \leq i< j \leq k}{\prod}q_{ij}^{-n_im_{j}}x^{m+n}\right)		
		=\underset{1 \leq i< j \leq k}{\prod}Q_{ij}^{-n_im_{j}+m_{ij}+n_{ij}}U_1^{m_1+n_1}\dotsc U_k^{m_k+n_k}=U(x^m)U(x^n)
	\end{split}
\end{equation*}
and so, $U$ is a unitary representation on $G_{dc}$. Let $x^m=\underset{1\leq i<j \leq k}{\prod}q_{ij}^{m_{ij}}s_1^{m_1}\dotsc s_k^{m_k} \in G_{dc}$ and let $h \in \HS$. Since $\underline{U}$ is a regular $Q$-unitary dilation of $\underline{T}$, we have
\begin{align*}
	P_{\HS}U(x^m)h
	&=P_\HS \underset{1\leq i<j \leq k}{\prod}Q_{ij}^{m_{ij}}U_1^{m_1}\dotsc U_k^{m_k}h\\
		&=P_\HS U_1^{m_1}\dotsc U_k^{m_k}\left(\underset{1\leq i<j \leq k}{\prod}Q_{ij}^{m_{ij}}h\right) \qquad [\text{since each $Q_{ij}$ commute with $U_1, \dotsc, U_k$}]\\
&=P_\HS U_1^{m_1}\dotsc U_k^{m_k}\left(\underset{1\leq i<j \leq k}{\prod}q_{ij}^{m_{ij}}h\right) \qquad [\text{as $Q_{ij}h=q_{ij}h$ for every $h \in \HS$}]\\
&=\underset{1\leq i<j \leq k}{\prod}q_{ij}^{m_{ij}}P_\HS U_1^{m_1}\dotsc U_k^{m_k}h\\
	&=\underset{1\leq i<j \leq k}{\prod}q_{ij}^{m_{ij}-m_{i}^+m_{j}^-}\left[(T_{1}^{m_1^-})^*\dotsc (T_{k}^{m_k^-})^*\right]\left[T_{1}^{m_1^+}\dotsc T_{k}^{m_k^+}\right]h \\
	&=T(x^m)h
\end{align*}
and thus, $P_{\HS}U(g)|_\HS=T(g)$ for all $g \in G_{dc}$. It follows from Theorem \ref{NaimarkI} that the $\mathcal{B}(\HS)$-valued map $T$ is positive definite. 

\medskip 	

Conversely, suppose $\underline{T}$ satisfies the Brehmer's positivity condition. It follows from Theorem \ref{DCIII} that the operator-valued function $T$ as in (\ref{Map2}) is positive definite. By Theorem \ref{NaimarkI}, there is a Hilbert space $\mathcal{K} \supseteq \mathcal{H}$ and a
	unitary representation $U: G_{dc} \to \mathcal{B}(\mathcal{K})$ such that
	$T(x^m)=P_\mathcal{H}U(x^m)|_\mathcal{H}$ for every $x^m \in G_{dc}$. Let  $U_i=U(s_i), Q_{ij}=U(q_{ij})$ and $Q_{ji}=Q_{ij}^*$ for $ 1 \leq i < j  \leq k$. Note that 
	\[
	U_iU_i^*=U(s_i)U(s_i^{-1})=U(e)=U(s_i^{-1})U(s_i)=U_i^*U_i
	\]
	and so, each $U_i$ is a unitary since $U(e)=I$. Similarly, each $Q_{ij}$ is also a unitary operator. Furthermore, it follows from the definition of a unitary representation that $U_iU_j=Q_{ij}U_jU_i$ and
	\begin{align*}
		\underset{1\leq i<j \leq k}\prod q_{ij}^{-m_{i}^+m_{j}^-}\left[(T_{1}^{m_1^-})^*\dotsc (T_{k}^{m_k^-})^*\right]\left[T_{1}^{m_1^+}\dotsc T_{k}^{m_k^+}\right]
		=T\left(s_1^{m_1}\dotsc s_k^{m_k}\right)
		=P_{\mathcal{H}}U_1^{m_1}\dotsc U_k^{m_k}|_\mathcal{H}
	\end{align*} 	
	for $m_1, \dotsc, m_k \in \mathbb{Z}$. Again by the fact that $U$ is a unitary representation of $T$, it follows that  
	\[
	q_{ij}I_\HS=P_{\mathcal{H}}Q_{ij}|_\mathcal{H}
	\quad \text{and so,} \quad 
	Q_{ij}=\begin{bmatrix}
		q_{ij}I_\HS & A_{ij}\\
		B_{ij} & \widetilde{Q}_{ij}
	\end{bmatrix}
	\]
	with respect to $\mathcal{K}=\mathcal{H}\oplus \mathcal{H}^{\perp}$. Since both $q_{ij}I_\HS$ and $Q_{ij}$ are unitaries, it is easy to see that  
	\[
	0=Q_{ij}^*Q_{ij}-I_{\mathcal{K}}=\begin{bmatrix}
		B_{ij}^*B_{ij} & * \\
		* & *
	\end{bmatrix} \quad \text{and} \quad 0=Q_{ij}Q_{ij}^*-I_{\mathcal{K}}=\begin{bmatrix}
		A_{ij}A_{ij}^* & * \\
		* & *
	\end{bmatrix}.
	\]
	Thus,
	$
	A_{ij}=0=B_{ij}
	$ 
	and so, $Q_{ij}|_\HS=q_{ij}I_\HS$ for $1 \leq i, j \leq k$ with $i \ne j$. The proof is complete.
\end{proof}

Though Naimark's theorem finds a way to establish an equivalence between Brehmer's positivity condition and regular $Q$-unitary dilation of a $q$-commuting finite tuple $\underline T$, clearly it fails to achieve a regular $q$-unitary dilation for $\underline T$. To resolve this issue, we employ Stienspring's dilation theorem which is stated below along with an elementary proposition from the literature that gives a connection between completely positive and completely bounded maps. 	
	\begin{prop}[\cite{Paulsen}, Proposition 3.6]\label{prop_cp_cb}
		Let $\mathcal{A}, \mathcal{B}$ be unital $C^*$-algebras and let $S \subseteq \mathcal{A}$ be an operator system. Then every completely positive map $\phi: S \to \mathcal{B}$ is completely bounded and $\|\phi(1)\|=\|\phi\|=\|\phi\|_{cb}$.
	\end{prop}

\begin{thm}[Stinespring's dilation theorem, \cite{Stine}]\label{Stine}
		Let $\mathcal{A}$ be a unital $C^*$-algebra and let $\HS$ be a Hilbert space. If $\phi: \mathcal{A} \to \mathcal{B}(\HS)$ is a unital completely positive map, then there exist a Hilbert space $\mathcal{K}$ containing $\HS$ and a unital $*$-homomorphism $\psi: \mathcal{A} \to \mathcal{B}(\mathcal{K})$ such that 
		$
		\phi(a)=P_\HS\psi(a)|_\HS
		$ 	
		for every $a \in \mathcal{A}$.	
	\end{thm}

 Indeed, we construct a unital $C^*$-algebra and a completely positive map corresponding to $\underline{T}$ to apply Stienspring's theorem as shown below. This settles the $(1)\Leftrightarrow (2)$ parts of the Theorem \ref{thm912} for a finite family of $q$-commuting contractions.

\begin{thm}\label{regular q-unitary dilationII}
	Let $\underline{T}=(T_1, \dotsc, T_k)$ be a $q$-commuting tuple of contractions with $\|q\|=1$ acting on a Hilbert space $\HS$. Then $\underline{T}$ admits a regular $q$-unitary dilation if and only if it satisfies the Brehmer's positivity condition.
\end{thm}

\begin{proof}
We begin with a $q$-commuting tuple $\underline{T}=(T_1, \dotsc, T_k)$ of contractions with $\|q\|=1$ acting on a Hilbert space $\HS$ which satisfies the Brehmer's positivity condition. Let $\mathbb{F}_k$ be the free group of reduced words consisting of letters $s_1, \dotsc, s_k$ and their inverses  $s_1^{-1}, \dotsc, s_k^{-1}$. Let $\C[\mathbb{F}_k]$ be the group algebra with elements  of the form 
${\displaystyle
\sum_{g \in \mathbb{F}_k}a_gg},
$
where each $a_g$ is a complex number and $a_g$ is non-zero for all but finitely many $g$ in $\mathbb{F}_k.$ The multiplication and involution operations are defined as
$
(a_gg)(a_hh):=a_ga_hgh$ and $(a_gg)^*=\overline{a_g}g^{-1}$.
The norm given by 
\[
\|.\|: \C[\mathbb{F}_k] \to [0, \infty), \quad \|a\|=\sup\{\|\pi(a)\|: \pi \ \mbox{is a representation of } \ \mathbb{C}[\mathbb{F}_k] \}
\]
is well-defined. The completion of $\mathbb{C}[\mathbb{F}_k]$ in this norm, denoted by $C^*(\mathbb{F}_k)$, is a unital $C^*$- algebra and contains $\mathbb{C}[\mathbb{F}_k]$ as a dense $*$-algebra. An interested reader may refer to Section 2.4 in \cite{PutnamII} for further details. Let $\mathcal{I}_{dc}$ be the closure of the two-sided ideal in $C^*(\mathbb{F}_k)$ generated by elements of the form
$
s_is_j-q_{ij}s_js_i$ and $s_is_j^{-1}-q_{ij}^{-1}s_j^{-1}s_i , \, \, (i \ne j)$. Thus 
$
s_is_j+\mathcal{I}_{dc}=q_{ij}s_js_i+\mathcal{I}_{dc}$ and $s_is_j^{-1}+\mathcal{I}_{dc}=q_{ij}^{-1}s_j^{-1}s_i+\mathcal{I}_{dc}$ for $i \ne j$.
Endowed with the natural quotient norm, the algebra $\mathcal{A}_{dc}$ given by
$
\mathcal{A}_{dc}=C^*(\mathbb{F}_k)\slash \mathcal{I}_{dc}
$ 
is a unital $C^*$-algebra. One can see Section 3.1 in \cite{Murphy} for further details on ideals and quotients of $C^*$-algebras. Let us take $S_{dc}$ to be an operator system of $\mathcal{A}_{dc}$ given by
$
S_{dc}=(\mathbb{C}[\mathbb{F}_k]+\mathcal{I}_{dc})\slash \mathcal{I}_{dc}.
$ 
The unit element of $S_{dc}$ is $e+\mathcal{I}_{dc}$, where $e$ (empty word) is the unit element of $\C[\mathbb{F}_k]$. Indeed, $S_{dc}$ is a dense $*$-subalgebra of $\mathcal{A}_{dc}$ since $\mathbb{C}[\mathbb{F}_k]$ is dense in $C^*(\mathbb{F}_k)$. Note that for any $g \in \mathbb{F}_k$ and $a_g \in \C$, the element $a_gg+\mathcal{I}_{dc}$ can be uniquely written as 

\begin{equation}\label{eqn8.24}
	a_gg+\mathcal{I}_{dc}=\alpha\underset{1\leq i<j \leq k}{\prod}q_{ij}^{m_{ij}}s_1^{m_1}\dotsc s_k^{m_k}+\mathcal{I}_{dc}=\alpha x^m+\mathcal{I}_{dc}
\end{equation}
for some $\alpha \in \C$ and $m_{ij}, m_1, \dotsc, m_k$ in $\Z$.  Any element $(f+\mathcal{I}_{dc})$ in $S_{dc}$ is a finite linear combination of elements of the form $a_gg+\mathcal{I}_{dc} \ (g \in G)$. Consider the map (which is extended linearly to $S_{dc}$) given by
\begin{equation}\label{eqn:phi2}
	\widehat{\phi}_{c}: S_{dc} \to \mathcal{B}(\HS), \quad 	\alpha x^m+\mathcal{I}_{dc} \mapsto \alpha T(x^m), 
\end{equation}
where $\displaystyle x^m=\underset{1\leq i<j \leq k}{\prod}q_{ij}^{m_{ij}}s_1^{m_1}\dotsc s_k^{m_k}$ and $T(x^m)$ is given as in \eqref{Map2}, i.e., 
	\[
T(x^m)=\underset{1\leq i<j \leq k}{\prod}q_{ij}^{m_{ij}}\underset{1\leq i<j \leq k}\prod q_{ij}^{-m_{i}^+m_{j}^-}\left[(T_{1}^{m_1^-})^*\dotsc (T_{k}^{m_k^-})^*\right]\left[T_{1}^{m_1^+}\dotsc T_{k}^{m_k^+}\right].
\]
We now show that $\widehat{\phi}_{c}$ is completely positive. It follows from Lemma 3.13 and discussion thereafter in \cite{Paulsen} that it suffices to show  
\[
\sum_{l=1}^{n}\sum_{t=1}^{n}\left \langle  \widehat{\phi}_{c}(a_l^*a_t)y_t, y_l \right \rangle \geq 0 
\]
for every $n \in \N, \ \{a_t \ :\  1 \leq t \leq n\} \subseteq S_{dc}$ and $\{y_t \ : \ 1 \leq t \leq n \} \subseteq \HS$. It follows from (\ref{eqn8.24}) that each $a_{t}$ can be written as a finite linear combination of the elements of the form
\[
\xi_{t}+\mathcal{I}_{dc}= \alpha_{t}\prod_{1 \leq i <j \leq k}q_{ij}^{m_{ij}(t)}s_1^{m_1(t)}\dotsc s_k^{m_k(t)}+\mathcal{I}_{dc}
\]
for some $\alpha_{t} \in \C$ and $m_{ij}(t), m_1(t), \dotsc, m_k(t)$ in $\Z$. Thus $\widehat{\phi}_{c}$ is completely positive if and only if

\[
\sum_{p,r=1}^{n}\left \langle  \widehat{\phi}_{c}((\xi_{p}+\mathcal{I}_{dc})^*(\xi_r+\mathcal{I}_{dc}))y_r, y_p \right \rangle \geq 0 
\]
for every $n \in \N, \ \{ \xi_1+\mathcal{I}_{dc}, \dotsc, \xi_n+\mathcal{I}_{dc}\} \subseteq  S_{dc}$ and $\{y_1, \dotsc, y_n\} \subseteq \mathcal{H}$. Note that
\begin{equation*}
	\begin{split}
		& \quad \ (\xi_p+\mathcal{I}_{dc})^*(\xi_r+\mathcal{I}_{dc})\\
		&=\bigg[\bigg(\alpha_p\prod_{1 \leq i <j \leq k}q_{ij}^{m_{ij}(p)}s_1^{m_1(p)}s_2^{m_2(p)}\dotsc s_k^{m_k(p)} \bigg)^*\alpha_r\prod_{1 \leq i <j \leq k}q_{ij}^{m_{ij}(r)}s_1^{m_1(r)}s_2^{m_2(r)}\dotsc s_k^{m_k(r)}\bigg]+\mathcal{I}_{dc}\\
		&=\bigg[\overline{\alpha}_p\alpha_r\prod_{1 \leq i <j \leq k}q_{ij}^{-m_{ij}(p)+m_{ij}(r)}s_k^{-m_k(p)}\dotsc s_2^{-m_2(p)}s_1^{-m_1(p)}s_1^{m_1(r)}s_2^{m_2(r)}\dotsc s_k^{m_k(r)}\bigg]+\mathcal{I}_{dc}
	\end{split}
\end{equation*}
and so,
\[
\sum_{p,r=1}^{n}\left \langle  \widehat{\phi}_{c}((\xi_p+\mathcal{I}_{dc})^*(\xi_r+\mathcal{I}_{dc}))y_r, y_p \right \rangle 
\]
\begin{small}
	\begin{equation*}
		\begin{split}
			&=\sum_{p,r=1}^{n}\left \langle  \widehat{\phi}_{c}\bigg(\overline{\alpha}_p\alpha_r\prod_{1 \leq i <j \leq k}q_{ij}^{-m_{ij}(p)+m_{ij}(r)}s_k^{-m_k(p)}\dotsc s_2^{-m_2(p)}s_1^{-m_1(p)}s_1^{m_1(r)}s_2^{m_2(r)}\dotsc s_k^{m_k(r)}+\mathcal{I}_{dc}\bigg)y_r, y_p \right \rangle \\		
			&= \sum_{p,r=1}^{n}\left \langle  \widehat{\phi}_{c}\bigg(s_k^{-m_k(p)}\dotsc s_2^{-m_2(p)}s_1^{-m_1(p)}s_1^{m_1(r)}s_2^{m_2(r)}\dotsc s_k^{m_k(r)}+\mathcal{I}_{dc}\bigg)\alpha_r \prod_{1 \leq i <j \leq k}q_{ij}^{m_{ij}(r)} y_r, \alpha_p \prod_{1 \leq i <j \leq k}q_{ij}^{m_{ij}(p)}y_p \right \rangle \\	
			&= \sum_{p,r=1}^{n}\left \langle  \widehat{\phi}_{c}((x^{m(p)})^{-1}x^{m(r)}+\mathcal{I}_{dc})y'_r, y'_p \right \rangle, \\	
		\end{split}
	\end{equation*}
\end{small}
where $x^{m(p)}=s_1^{m_1(p)}\dotsc s_k^{m_k(p)}$ in $\mathbb{C}[\mathbb{F}_k]$ and $y_p'= \alpha_p \underset{1 \leq i <j \leq k}{\prod}q_{ij}^{m_{ij}(p)}y_p$ in $\mathcal{H}$. Define 
\[
h(x^{m(p)})=\left\{
\begin{array}{ll}
	y_p' & \mbox{if} \ 1\leq p \leq n\\
	0 &  \mbox{otherwise} 
\end{array} 
\right. 
\]
which is a function in $c_{00}(G_{dc}, \mathcal{H})$ with respect to the group $G_{dc}$ defined earlier. Also, note that
\begin{equation*}
	\begin{split} 
		\widehat{\phi}_{c}((x^{m(p)})^{-1}x^{m(r)}+\mathcal{I}_{dc})&= T((x^{m(p)})^{-1}x^{m(r)}).
	\end{split}
\end{equation*}
Putting everything together, we have
\begin{align}\label{eqn9.019}
		\sum_{p,r=1}^{n}\left \langle  \widehat{\phi}_{c}((\xi_p+\mathcal{I}_{dc})^*(\xi_r+\mathcal{I}_{dc}))y_r, y_p \right \rangle
		&=\sum_{p,r=1}^{n}\left \langle  \widehat{\phi}_{c}((x^{m(p)})^{-1}x^{m(r)}+\mathcal{I}_{dc})y'_r, y'_p \right \rangle \notag \\
		&=\sum_{p,r=1}^{n}\left \langle  T((x^{m(p)})^{-1}x^{m(r)})h(x^{m(r)}), h(x^{m(p)}) \right \rangle,
	\end{align}
	which is non-negative due to Theorem \ref{DCIII} since $\underline{T}$ satisfies the Brehmer's positivity condition. Hence, $\widehat{\phi}_{c}$ is a completely positive map on the operator system $S_{dc}$. It follows from Proposition \ref{prop_cp_cb} that $\widehat{\phi}_{c}$ is completely bounded and so,  
\begin{equation}\label{eqn8.25}
	\|\widehat{\phi}_{c}\|=\|\widehat{\phi}_{c}(e+\mathcal{I}_{dc})\|=1.
\end{equation}
Since the norm closure of $S_{dc}$ is $\mathcal{A}_{dc}$, we can extend the map $\widehat{\phi}_{c}$ continuously to $\mathcal{A}_{dc}$.  Indeed, if $f+\mathcal{I}_{dc} \in \mathcal{A}_{dc}$, then $f \in C^*(\mathbb{F}_k)$ and so, there exists a sequence $\{f_n\}_{n \in \N} \subseteq \C[\mathbb{F}_k]$ such that $\|f_n-f\|_{C^*(\mathbb{F}_k)} \to 0$ as $n\to \infty$. It follows from the definition of the quotient norm that 
\[
\|(f_n+\mathcal{I}_{dc})-(f+\mathcal{I}_{dc})\|_{\mathcal{A}_{dc}} \leq \|f_n-f\|_{C^*(\mathbb{F}_k)} \to 0 \quad \text{as} \ n \to \infty.
\]
Indeed, $\{f_n+\mathcal{I}_{dc}\}$ is a Cauchy sequence in $\mathcal{A}_{dc}$. Moreover, we have 
\[
\|\widehat{\phi}_{c}(f_n+\mathcal{I}_{dc})-\widehat{\phi}_{c}(f_m+\mathcal{I}_{dc})\|=\|\widehat{\phi}_{c}(f_n-f_m+\mathcal{I}_{dc})\| \leq \|f_n-f_m+\mathcal{I}_{dc}\| \to 0 \quad \text{as} \ n, m \to \infty,
\]
where the last inequality follows from (\ref{eqn8.25}). Thus $\{\widehat{\phi}_{c}(f_n+\mathcal{I}_{dc})\}$ is a Cauchy sequence in the Banach space $\mathcal{B}(\HS)$ and so, there is a unique operator, say $\widehat{\phi}_{c}(f+\mathcal{I}_{dc})$,  acting on $\HS$ such that 
\[
\|\widehat{\phi}_{c}(f+\mathcal{I}_{dc})-\widehat{\phi}_{c}(f_n+\mathcal{I}_{dc})\| \to 0 \quad \text{as} \ n \to \infty. 
\]
Indeed, by limiting criterion, we get 
\begin{equation}\label{eqn8.26}
	\|\widehat{\phi}_{c}(f+\mathcal{I}_{dc})\|=\lim_{n \to \infty}\|\widehat{\phi}_{c}(f_n+\mathcal{I}_{dc})\| \leq \lim_{n \to \infty} \|f_n+\mathcal{I}_{dc}\|_{\mathcal{A}_{dc}} = \|f+\mathcal{I}_{dc}\|_{\mathcal{A}_{dc}}.
\end{equation}
It follows from (\ref{eqn8.26}) and the uniqueness of the limit $\widehat{\phi}_{c}(f+\mathcal{I}_{dc})$ for every $f+\mathcal{I}_{dc}$ in $\mathcal{A}_{dc}$ that $\widehat{\phi}_{c}$ extends to a contractive linear map on $\mathcal{A}_{dc}$. By continuity arguments, $\widehat{\phi}_{c}$ extends to a unital completely positive map on $\mathcal{A}_{dc}$. It now follows from Theorem \ref{Stine} that there is a larger Hilbert space $\mathcal{K}$ containing $\mathcal{H}$ and a unital $*$-algebra homomorphism 
$
\psi:\mathcal{A}_{dc} \to \mathcal{B}(\mathcal{K})$ such that
$\widehat{\phi}_{c}(a)=P_{\mathcal{H}}\psi(a)|_\mathcal{H}$, for every $ a \in \mathcal{A}_{dc}$. Define $U_i=\psi(s_i+\mathcal{I}_{dc})$ for $i=1, \dotsc, k$. Since $\psi$ is an algebra $*$-homomorphism, $(U_1, \dotsc, U_k)$ is a tuple of unitaries such that 
$
U_iU_j=q_{ij}U_jU_i$ and $U_iU_j^{-1}=q_{ij}^{-1}U_j^{-1}U_i$, for $1 \leq i < j \leq k$. Moreover, for every $m_1, \dotsc, m_k \in \mathbb{Z}$, we have 
\begin{equation*}
	\begin{split} 
		P_\mathcal{ H }U_1^{m_1}\dotsc U_k^{m_k}|_\mathcal{ H }
		&=P_\mathcal{ H }\psi(s_1^{m_1}\dotsc s_k^{m_k} + \mathcal I_{dc})|_\mathcal{H} \\
		&=\widehat{\phi}_{c}(s_1^{m_1}\dotsc s_k^{m_k} + \mathcal I_{dc})\\
		&=\underset{1\leq i<j \leq k}\prod q_{ij}^{-m_{i}^+m_{j}^-}\left[(T_{1}^{m_1^-})^*\dotsc (T_{k}^{m_k^-})^*\right]\left[T_{1}^{m_1^+}\dotsc T_{k}^{m_k^+}\right].
	\end{split}
\end{equation*} 
Therefore, $\underline{T}$ admits a regular $q$-unitary dilation and the proof is complete.
\end{proof}

Below we find a few classes of $q$-commuting contractions with $\|q\|=1$ which satisfy the Brehmer's positivity condition.

\begin{lem}\label{lem908}
		Let $\underline{T}=(T_1, \dotsc, T_k)$ be a $q$-commuting tuple of contractions with $\|q\|=1$ on a Hilbert space $\HS$. For $ u \subseteq \{1, \dotsc, k\}$, we have that $S(u)\geq 0$ (with $S(u)$ as in \eqref{Brehmer's}) in each of the following cases:
		\begin{enumerate}
			\item[(i)] $\underline{T}$ consists of isometries;
			\item[(ii)] $\underline{T}$ consists of doubly $q$-commuting contractions;
			\item[(iii)] $\|T_1h\|^2 + \dotsc +\|T_kh\|^2  \leq \|h\|^2$ for all $h \in \HS$.
		\end{enumerate}
	\end{lem}
	
	\begin{proof}
			Let $\underline{T}$ be a $q$-commuting tuple of isometries. Then 
		\[
		S(u)=\sum_{v \subset u }(-1)^{|v|}T(x^{e(v)})^*T(x^{e(v)})=\underset{v \subset u}\sum(-1)^{|v|}I \geq 0,
		\]
		where $T(x^{e(v)})$ is as in Theorem \ref{DCIII}. The last inequality follows because $\underset{v \subset u}\sum(-1)^{|v|}$ is either $0$ or $1$ depending on the choice of $u$. Let $(T_1, \dotsc, T_k)$ be a doubly $q$-commuting tuple of contractions. We note that Proposition 3.2 in \cite{Barik} establishes that a doubly $q$-commuting tuple of contractions satisfies Szeg\"{o} positivity, where the adjoints are considered on the right, in contrast to Brehmer's positivity condition, which involves adjoints on the left. Nonetheless, the underlying idea of the proof remains the same. We briefly discuss the proof here for the sake of completeness. Let $v=\{n_1,, \dotsc, n_m\}$ be any non-empty subset of $\{1,\dotsc, k\}$ with $n_1 < \dotsc < n_m$. Then $T(x^{e(v)})^*T(x^{e(v)})$
		\begin{equation*}
			\begin{split}
				&=(T_{n_1}^{e_{n_{1}}(v)}T_{n_2}^{e_{n_{2}}(v)}\dotsc T_{n_m}^{e_{n_m}(v)})^*(T_{n_1}^{e_{n_{1}}(v)}T_{n_2}^{e_{n_{2}}(v)}\dotsc T_{n_m}^{e_{n_m}(v)})\\
				&=(T_{n_m}^{e_{n_m}(v)})^*\dotsc (T_{n_2}^{e_{n_{2}}(v)})^* (T_{n_1}^{e_{n_{1}}(v)})^*T_{n_1}^{e_{n_{1}}(v)}T_{n_2}^{e_{n_{2}}(v)}\dotsc T_{n_m}^{e_{n_m}(v)}\\
				&=\prod_{1 \leq i <m}q_{n_{m}n_{i}}^{e_{n_m}(v)e_{n_{i}}(v)}(T_{n_{m-1}}^{e_{n_{m-1}}(v)})^*\dotsc (T_{n_2}^{e_{n_{2}}(v)})^* (T_{n_1}^{e_{n_{1}}(v)})^*(T_{n_m}^{e_{n_m}(v)})^*T_{n_1}^{e_{n_{1}}(v)}T_{n_2}^{e_{n_{2}}(v)}\dotsc T_{n_m}^{e_{n_m}(v)}\\
				&=\prod_{1 \leq i <m}q_{n_{m}n_{i}}^{e_{n_m}(v)e_{n_{i}}(v)}\prod_{1 \leq i <m}q_{n_{m}n_i}^{-e_{n_m}(v)e_{n_{i}}(v)}(T_{n_{m-1}}^{e_{n_{m-1}}(v)})^*\dotsc (T_{n_1}^{e_{n_{1}}(v)})^*T_{n_1}^{e_{n_{1}}(v)}T_{n_2}^{e_{n_{2}}(v)}\dotsc (T_{n_m}^{e_{n_m}(v)})^*T_{n_m}^{e_{n_m}(v)}\\
				&=\bigg[(T_{n_{m-1}}^{e_{n_{m-1}}(v)})^*\dotsc (T_{n_2}^{e_{n_{2}}(v)})^* (T_{n_1}^{e_{n_{1}}(v)})^*T_{n_1}^{e_{n_{1}}(v)}T_{n_2}^{e_{n_{2}}(v)}\dotsc T_{n_{m-1}}^{e_{n_{m-1}}(v)}\bigg] (T_{n_m}^{e_{n_m}(v)})^*T_{n_m}^{e_{n_m}(v)}\\
				&=\vdots \ \quad \mbox{(continuing doing the previous steps for $i=n_1, \dotsc, n_{m-1}$)}\\
				&=\prod_{1 \leq i \leq m}(T_{n_{i}}^{e_{n_{i}}(v)})^*(T_{n_{i}}^{e_{n_{i}}(v)})\\
				&=\prod_{\omega \in v}T_{\omega}^*T_{\omega}.
			\end{split}
		\end{equation*}
		Thus for any subset $u$ of $\{1, \dotsc, k\},$ we have that 
		\begin{equation}\label{eqn819}
			\begin{split}
				S(u)=\sum_{v \subset u}(-1)^{|v|}T(x^{e(v)})^*T(x^{e(v)})
				=\sum_{v \subset u}(-1)^{|v|}\prod_{\omega \in v}T_{\omega}^*T_{\omega}.\\
			\end{split}
		\end{equation}	 
		The product $\prod_{\omega \in v}T_{\omega}^*T_{\omega}$ is well-defined because for any doubly $q$-commuting pair $(A,B)$ with $|q|=1$ the pair $(A^*A, B^*B)$ is a doubly commuting pair. Finally (\ref{eqn819}) gives that
		\[
		S(u)=\sum_{v \subset u}(-1)^{|v|}\prod_{\omega \in v}T_{\omega}^*T_{\omega}=\prod_{\omega \in u}(I-T_{\omega}^*T_{\omega}).
		\]
		Since $I-T_{\omega}^*T_{\omega} (\omega \in u)$ commute with each other and are non-negative, we have that $S(u)\geq 0$.
		
		\medskip	
		
		Let $\|T_1h\|^2 + \dotsc +\|T_kh\|^2  \leq \|h\|^2$ for all $h \in \HS$. We follow the proof of Proposition 2 in \cite{Attele}. Let $u=\{\omega_1, \dotsc ,\omega_r\}$ and write $T_i$ in place of $T_{\omega_i}$ for the ease of computations. For $0 \leq p \leq r$ and $h \in \mathcal{H}$, let us define  
		\[
		a_p(h)=\sum_{\substack{v \subset u \\ |v|=p}}\|T(x^{e(v)})h\|^2. 
		\]
		Take any permutation $\sigma$ on $\{1, \dotsc, k\}$. Then 
		\begin{equation*}
			\begin{split}
				T\left(q^{m_0}s_{\sigma(1)}^{m_{\sigma(1)}}\dotsc s_{\sigma(k)}^{m_{\sigma(k)}}\right)
				&=q^{m_0}\underset{1\leq i<j \leq k}\prod q_{\sigma(i)\sigma(j)}^{-m_{\sigma(i)}^+m_{\sigma(j)}^-}\left[(T_{\sigma(1)}^{m_{\sigma(1)}^-})^*\dotsc (T_{\sigma(k)}^{m_{\sigma(k)}^-})^*\right]\left[T_{\sigma(1)}^{m_{\sigma(1)}^+}\dotsc T_{\sigma(k)}^{m_{\sigma(k)}^+}\right]\\
				&=\underset{1\leq i<j \leq k}\prod q_{ij}^{\alpha_i\beta_j}\bigg( \underset{1\leq i<j \leq k}\prod q_{ij}^{-m_{i}^+m_{j}^-} \, q^{m_0}\left[(T_{1}^{m_1^-})^*\dotsc (T_{k}^{m_k^-})^*\right]\left[T_{1}^{m_1^+}\dotsc T_{k}^{m_k^+}\right]\bigg)\\
				&=\underset{1\leq i<j \leq k}\prod q_{ij}^{\alpha_i\beta_j}T(q^{m_0}s_1^{m_1}\dotsc s_k^{m_k})\\	
			\end{split}
		\end{equation*} 
		for some $\alpha_i, \beta_j \in \Z$. Therefore, we have for all $h \in \HS$ that
		\[
		\left\|T(q^{m_0}s_1^{m_1}\dotsc s_k^{m_k})h\right\|=\left\|T(q^{m_0}s_{\sigma(1)}^{m_{\sigma(1)}}\dotsc s_{\sigma(k)}^{m_{\sigma(k)}})h\right\|.
		\]	
		Let $p \in \{1, \dotsc, r\}$ and let $h \in \HS$. Some routine computations give that
		\begin{equation*}
			\begin{split}
				\overset{r}{\underset{j=1}{\sum}}  \sum_{\substack{v \subset u \\ |v|=p-1}}\|T_jT(x^{e(v)})h\|^2
				&= \overset{r}{\underset{j=1}{\sum}} \quad  \sum_{\substack{n_1+\dotsc+n_r=p-1 \\ n_1, \dotsc, n_r\in \{0,1\}  }}\|T_jT_1^{n_1}\dotsc T_j^{n_j}\dotsc T_r^{n_r}h\|^2\\
				& =\overset{r}{\underset{j=1}{\sum}} \quad  \sum_{\substack{n_1+\dotsc+n_r=p-1 \\ n_1, \dotsc, n_r\in \{0,1\}  }}\|T_1^{n_1}\dotsc T_j^{n_j+1}\dotsc T_r^{n_r}h\|^2\\
				& \geq \overset{r}{\underset{j=1}{\sum}} \quad  \sum_{\substack{n_1+\dotsc+n_r=p-1 \\ n_1, \dotsc, n_r\in \{0,1\}, n_j =0  }}\|T_1^{n_1}\dotsc T_j^{n_j+1} \dotsc T_r^{n_r}h\|^2\\
				&\geq  \sum_{\substack{v \subset u \\ |v|=p}}\|T(x^{e(v)})h\|^2\\
				&=a_p(h).
			\end{split}
		\end{equation*}
		Consequently, we have 
		\begin{equation*}
			\begin{split}
				a_p(h)=\sum_{\substack{v \subset u \\ |v|=p}}\|T(x^{e(v)})h\|^2
				\leq  \overset{r}{\underset{j=1}{\sum}} \left( \sum_{\substack{v \subset u \\ |v|=p-1}}\|T_jT(x^{e(v)})h\|^2\right)
				& =   \sum_{\substack{v \subset u \\ |v|=p-1}} \left(\overset{r}{\underset{j=1}{\sum}}\|T_jT(x^{e(v)})h\|^2\right)\\		
				& \leq   \sum_{\substack{v \subset u \\ |v|=p-1}} \|T(x^{e(v)})h\|^2\\		
				&=a_{p-1}(h),
			\end{split}
		\end{equation*}
		and hence
		\begin{equation*}
			\begin{split}
				\langle S(u)h,h\rangle =\sum_{v \subset u}(-1)^{|v|}\|T(x^{e(v)})h\|^2=\overset{r}{\underset{p=0}{\sum}}(-1)^pa_{p}(h)
				\geq a_0(h)-a_1(h)
				=\|h\|^2-\sum_{i=1}^r\|T_ih\|^2
				\geq 0.\\
			\end{split}
		\end{equation*}
		The proof is now complete. 
	\end{proof}
	
	Next, we present a proof of Theorem \ref{dilation of DCV} for a finite family of $q$-commuting contractions with $\|q\|=1$.
	
	\begin{thm}\label{dilation of DCIV}
		Let  $\underline{T}=(T_1, \dotsc, T_k)$ be a $q$-commuting tuple of contractions with $\|q\|=1$ acting on a Hilbert space $\mathcal{H}$. Then $\underline{T}$ admits a regular $q$-unitary dilation in each of the cases below:
		\begin{enumerate}
			\item $\underline{T}$ consists of isometries;
			\item $\underline{T}$ consists of doubly $q$-commuting contractions;
			\item $\|T_1h\|^2+\dotsc + \|T_kh\|^2 \leq \|h\|^2$ for all $h \in \HS$.
		\end{enumerate}
	\end{thm}
	
	\begin{proof}
For each of the classes as in the statement of the theorem, $\underline{T}$ satisfies Brehmer's positivity condition which follows from Lemma \ref{lem908}. The desired conclusion follows from Theorem \ref{regular q-unitary dilationII}.	
	\end{proof}
	
	We now present an analog of von Neumann's inequality for the aforementioned classes of $q$-commuting contractions with $\|q\|=1$. Recall that $\mathcal{A}_{dc}=C^*(\mathbb{F}_k)\slash \mathcal{I}_{dc}$ is a unital $C^*$-algebra as discussed in the proof of Theorem \ref{regular q-unitary dilationII}.
	\begin{thm}\label{thm_vNIII}
		Let $\underline{T}=(T_1, \dotsc, T_k)$ be a $q$-commuting tuple of contractions with $\|q\|=1$ acting on a Hilbert space $\HS$ such that one of the following holds:
		\begin{enumerate}
			\item $\underline{T}$ consists of isometries;
			\item $\underline{T}$ consists of doubly $q$-commuting contractions;
			\item $\|T_1h\|^2+\dotsc + \|T_kh\|^2 \leq \|h\|^2$ for all $h \in \HS$.
		\end{enumerate}
Then for every $f_{ij} \in \mathcal{A}_{dc}$ and $n \in \N$, we have 
		\[
		\left\|\left[\widehat{\phi}_{c}(f_{ij})\right]_{i,j=1}^n\right\|_{M_n(\mathcal{B}(\HS))} \leq 	\left\|\left[f_{ij}\right]_{i,j=1}^n\right\|_{M_n(\mathcal{A}_{dc})},
		\]
where  $\widehat{\phi}_{c}$ is the map as in $($\ref{eqn:phi2}$)$. 		
	\end{thm} 
	
	\begin{proof}
		It follows from the proof of Theorem \ref{regular q-unitary dilationII} that the operator-valued map $\widehat{\phi}_{c}$ on  $S_{dc}$ given by
		\begin{equation*}
			\widehat{\phi}_{c}: S_{dc} \to \mathcal{B}(\HS), \quad \alpha x^m+\mathcal{I}_{dc} \mapsto \alpha T(x^m), 
		\end{equation*}
		is a completely contractive map. By continuity argument, it follows that $\widehat{\phi}_{c}$ is a completely contractive map on $\mathcal{A}_{dc}$, which completes the proof.  
	\end{proof}

Let $\underline{T}=(T_1, \dotsc, T_k)$ be a $q$-commuting tuple of contractions with $\|q\|=1$, belonging to one of the classes described in Theorem \ref{thm_vNIII}. We follow the same notations as in Theorems \ref{regular q-unitary dilationII} and \ref{thm_vNIII}. If each $q_{ij}=1$, then $\underline{T}$ is a commuting tuple. In this case, one replaces the free group $\mathbb{F}_k$ with $\mathbb{Z}^k$ in the proof of Theorem \ref{regular q-unitary dilationII}. The closed two-sided ideal $\mathcal{I}_{dc}$ becomes $\{0\}$. Thus, the $*$-algebra $S_{dc}=(\mathbb{C}[\mathbb{F}_k]+\mathcal{I}_{dc})\slash \mathcal{I}_{dc}$ is simply the commutative group algebra $\mathbb{C}[\Z^k]$. Its completion $\mathcal{A}_{dc}$ is the group $C^*$-algebra $C^*(\Z^k)$. Let $\{e_1, \dotsc, e_k\}$ be the standard generators in $\Z^k$. By Theorem 2.5.5 in \cite{PutnamII}, there is an isometric isomorphism $\varphi: \mathcal{A}_{dc} \to C(\T^k)$ with $\varphi(e_j)=z_j$ for $1 \leq j \leq k$. Hence, $\widehat{\phi}_c(\varphi^{-1}(p))=p(T_1, \dotsc, T_k)$ for every polynomial $p$ in $k$-variables. By Theorem \ref{thm_vNIII}, we have
\[
\|p(T_1, \dotsc, T_k)\|=\|\widehat{\phi}_c(\varphi^{-1}(p))\| \leq \|\varphi^{-1}(p)\|_{\mathcal{A}_{dc}}=\|p\|_{\infty, \T^k}
\]
for every polynomial $p$ in $k$-variables, which is precisely the multivariate von Neumann's inequality in the commutative setting.
	
	\section{Regular $q$-unitary dilation: The general case}\label{sec10}
	
	\noindent In this Section, we settle the proofs of our main results, Theorems \ref{thm912} and \ref{dilation of DCV} for any $q$-commuting family of contractions with $\|q\|=1$, i.e., the general case. Following the proofs of Theorems \ref{DCIII} and \ref{regular q-unitary dilation}, one can expect to extend these two theorems in the setting where the unimodular scalars are replaced by commuting unitaries. So, we first establish these theorems in the general framework.
	
\subsection{Regular $\widetilde{Q}$-unitary dilation for $Q$-commuting contractions} For extending Theorems \ref{DCIII} and \ref{regular q-unitary dilation} to a $Q$-commuting family, we need to associate a group with that family. We construct such a group following Lemma \ref{group}. Let  $\{T_\alpha: \alpha \in \Lambda\}$ be a $Q$-commuting family of contractions acting on a Hilbert space $\HS$ equipped with an order, say  $`` \preceq "$ on $\Lambda$. Consider the set given by 
	\begin{small}\begin{equation*}
			G_{\Lambda}=\left\{ x^m=\underset{1\leq i < j \leq k}{\prod}q_{\alpha_i \alpha_j}^{m_{\alpha_i\alpha_j}}s_{\alpha_1}^{m_{\alpha_1}}\dotsc s_{\alpha_k}^{m_{\alpha_k}}\ : \ \alpha_1 \preceq \dotsc \preceq \alpha_k \ \text{in $\Lambda$}, \ \  m_{\alpha_i \alpha_j}, \ m_{\alpha_1},\dotsc, m_{\alpha_k} \in \Z \ \ \text{and} \ \ k \in \N \cup \{0\} \right\}.
		\end{equation*}
	\end{small}
	We assume that the indeterminates $q_{\alpha_i\alpha_j}$ commute with $q_{\beta_i \beta_j}$ as well as with every $s_{\alpha_k}$ and we have
	\begin{align*}
		s_{\alpha_i}s_{\alpha_j}=\left\{
		\begin{array}{ll}
			q_{\alpha_i\alpha_j}s_{\alpha_j}s_{\alpha_i}, & \alpha_i \preceq \alpha_j\\
			q_{\alpha_j\alpha_i}^{-1}s_{\alpha_j}s_{\alpha_i}, & \alpha_j \preceq \alpha_i 
		\end{array} 
		\right.,  
		\quad 
		s_{\alpha_i}s_{\alpha_j}^{-1}=\left\{
		\begin{array}{ll}
			q_{\alpha_i \alpha_j}^{-1}s_{\alpha_j}^{-1}s_{\alpha_i}, & \alpha_i \preceq \alpha_j\\
			q_{\alpha_j \alpha_i}s_{\alpha_j}^{-1}s_{\alpha_i}, & \alpha_j \preceq \alpha_i 
		\end{array} 
		\right.
	\end{align*}
	for every $\alpha_i, \alpha_j$ in $\Lambda$ with $\alpha_i \ne \alpha_j$. Following the similar computations as in Lemma \ref{group}, one can easily show that $G_\Lambda$ is a group. With not much difficulty, one can obtain the next result by simply following the proof of Theorem \ref{DCIII}.
	
	\begin{thm}\label{thm_506}
		Let  $\mathcal{T}=\{T_\alpha: \alpha \in \Lambda\}$ be a $Q$-commuting family of contractions acting on a Hilbert space $\HS$. The map $T: G_{\Lambda} \to \mathcal{B}(\mathcal{H})$ defined by 
		\begin{small} 	\begin{align*}
				\begin{split} 
					T\bigg(\underset{1\leq i < j \leq k}{\prod}q_{\alpha_i \alpha_j}^{m_{\alpha_i\alpha_j}}s_{\alpha_1}^{m_{\alpha_1}}\dotsc s_{\alpha_k}^{m_{\alpha_k}}\bigg)
					=\underset{1\leq i<j \leq k}{\prod}Q_{\alpha_i\alpha_j}^{m_{\alpha_i \alpha_j}}\underset{1\leq i<j \leq k}\prod Q_{\alpha_i \alpha_j}^{-m_{\alpha_ i}^+m_{\alpha_ j}^-}\left[(T_{\alpha_1}^{m_{\alpha_1}^-})^*\dotsc (T_{\alpha_k}^{m_{\alpha_k}^-})^*\right]\left[T_{\alpha_1}^{m_{\alpha_1}^+}\dotsc T_{\alpha_k}^{m_{\alpha_k}^+}\right]
				\end{split}
			\end{align*} 
		\end{small}	
		is positive definite if and only if the family $\mathcal{T}$ satisfies the Brehmer's positivity condition. 
	\end{thm}

	\begin{defn} 
		A $Q$-commuting family $\{T_\alpha: \alpha \in \Lambda\}$ of contractions on a Hilbert space $\HS$ is said to have a \textit{regular $\widetilde{Q}$-unitary dilation} if there exist a Hilbert space $\mathcal
		{K} \supseteq \HS$ and a $\widetilde{Q}$-commuting family $\mathcal{U}=\{U_\alpha : \alpha\in \Lambda \}$ of unitaries on $\mathcal{K}$ such that $\widetilde{Q}_{\alpha \beta}|_{\HS}=Q_{\alpha \beta}$ for all $\alpha, \beta \in \Lambda$ with $\alpha \ne \beta$ and 
		\[
		\underset{1\leq i<j \leq k}\prod Q_{\alpha_{i}\alpha_{j}}^{-m_{\alpha_{i}}^+m_{\alpha_{j}}^-}\left[(T_{\alpha_{1}}^{m_{\alpha_{1}}^-})^*\dotsc (T_{\alpha_{k}}^{m_{\alpha_{k}}^-})^*\right]\left[T_{\alpha_{1}}^{m_{\alpha_{1}}^+}\dotsc T_{\alpha_{k}}^{m_{\alpha_{k}}^+}\right]=P_\mathcal{H}U_{\alpha_{1}}^{m_{\alpha_{1}}}\dotsc U_{\alpha_{k}}^{m_{\alpha_{k}}}|_\mathcal{H}
		\]
		for every $m_{\alpha_1}, \dotsc, m_{\alpha_k} \in \mathbb{Z}$ and $\alpha_1, \dotsc, \alpha_k \in \Lambda$ with $\alpha_1 \preceq \dotsc \preceq \alpha_k$.		
	\end{defn}
	
Now we present the desired generalization of Theorem \ref{regular q-unitary dilation}.
	
	\begin{thm}\label{thm_408}
		A $Q$-commuting family of contractions  $\mathcal{T}=\{T_\alpha: \alpha \in \Lambda\}$ acting on a Hilbert space $\HS$ has a regular $\widetilde{Q}$-unitary dilation if and only if it satisfies the Brehmer's positivity condition.
	\end{thm}
	
	\begin{proof}
In the proof of Theorem \ref{regular q-unitary dilation}, if we replace the group $G_{dc}$ by $G_\Lambda$ and the associated operator-valued function on $G_{dc}$ by the operator-valued map $T$ on $G_\Lambda$ as in Theorem \ref{thm_506}, then similar computations as in Theorem \ref{regular q-unitary dilation} lead to the conclusion.	
	\end{proof}
	
	As an application of Theorem \ref{thm_408}, we have the following result. 
	
	\begin{thm}
		A $Q$-commuting family of contractions $\mathcal{T}=\{T_\alpha : \alpha \in \Lambda\}$ acting on a Hilbert space $\HS$ admits a regular $\widetilde{Q}$-unitary dilation in each of the following cases:
		\begin{enumerate}
			\item $\mathcal{T}$ consists of isometries;
			\item $\mathcal{T}$ consists of doubly $Q$-commuting contractions;
			\item $\mathcal{T}$ is a countable family and $\sum_{\alpha \in \Lambda} \|T_{\alpha}h\|^2 \leq \|h\|^2$ for all $h \in \HS$. 
		\end{enumerate}	
		Moreover, if $\mathcal{T}$ consists of isometries, then there exists a  $\widetilde{Q}$-commuting family $\mathcal{U}=\{U_\alpha : \alpha \in \Lambda\}$ of unitaries on a larger space $\mathcal{K}$ containing $\HS$ such that $V_\alpha=U_\alpha|_{\HS}$ for every $\alpha \in \Lambda$.
	\end{thm}
	
	\begin{proof}
		Each of the classes mentioned in the statement of the theorem satisfies Brehmer's positivity condition and its proof runs along the same lines as Lemma \ref{lem908}. The desired conclusion now follows from Theorem \ref{thm_408}.
	\end{proof}

	We now prove our main results, i.e., Theorem \ref{thm912} and Theorem \ref{dilation of DCV} in general setting.

	\medskip

	\noindent \textbf{Proof of Theorem \ref{thm912}.} Let $\mathcal{T}=\{T_\alpha : \alpha \in \Lambda\}$ be a $q$-commuting family of contractions with $\|q\|=1$ acting on a Hilbert space $\HS$. The part $(3) \implies (2)$ follows from Theorem \ref{thm_408} and $(1) \implies (3)$ follows trivially. It suffices to prove the part $(2) \implies (1)$. Assume that $S(u) \geq 0$ for every finite subset $u$ of $\Lambda$.  Put an order ``$\preceq$'' on $\Lambda$ so that $(\Lambda, \preceq)$ is well-ordered.  Let $\mathbb{F}_{\infty}$ be the free group of reduced words consisting of letters $\{s_\alpha, s_\alpha^{-1}: \alpha \in \Lambda\}$. Each element in the group $\mathbb{F}_\infty$ is a word consisting of finitely many $s_\alpha, s_\alpha^{-1}$ and the empty word $e$ is the identity element. Let $\mathcal{I}_{\infty}$ be the closure of the two sided closed ideal in the group $C^*$-algebra $C^*(\mathbb{F}_\infty)$ which is generated by elements of the form 
	$	s_\alpha s_\beta-q_{\alpha \beta}s_\beta s_\alpha$ and $s_\alpha s_\beta^{-1}-q_{\alpha \beta}^{-1}s_\beta^{-1} s_\alpha$ for $\alpha, \beta \in \Lambda$ with $\alpha \ne \beta$.	One can employ the similar arguments as in the proof of Theorem \ref{regular q-unitary dilationII} for the quotient group $C^*$-algebra $\mathcal{A}_{\infty}=C^*(\mathbb{F}_\infty)\slash \mathcal{I}_{\infty}$ to obtain the regular $q$-unitary dilation of $\mathcal{T}$. We define $S_{\infty}$ to be the operator system of $\mathcal{A}_{\infty}$ given by
	\[
	S_{\infty}:=(\mathbb{C}[\mathbb{F}_\infty]+\mathcal{I}_{\infty})\slash \mathcal{I}_{\infty}.
	\] 
	The unit element of $S_{\infty}$ is $e+\mathcal{I}_{\infty}$ and $S_{\infty}$ is a dense $*$-subalgebra of $\mathcal{A}_{\infty}$. An element in $S_{\infty}$ can be uniquely written as a finite linear combination of the elements of the form
	
	\begin{equation}\label{eqn920}
		e+\mathcal{I}_\infty, \quad  s_{i_1}^{m_{i_1}}s_{i_2}^{m_{i_2}}\dotsc s_{i_t}^{m_{i_t}}+\mathcal{I}_{\infty}
	\end{equation}
	for $\{i_1, \dotsc, i_t\} \subset \Lambda$ with $i_1 \preceq \dotsc \preceq i_t$ and $\{m_{i_1}, \dotsc, m_{i_t}\} \subset \Z \setminus \{0\}$. Let us define
	\begin{equation}
		\widehat{\phi}_{\infty}: S_{\infty} \to \mathcal{B}(\HS), \quad 	\left(\alpha_0 e+\underset{1 \leq t \leq \ell}{\sum}\alpha_t s_{i_1}^{m_{i_1}}\dotsc s_{i_t}^{m_{i_t}}\right)+\mathcal{I}_{\infty} \mapsto \alpha_0I+\underset{1 \leq t \leq \ell}{\sum}\alpha_t T(s_{i_1}^{m_{i_1}}\dotsc s_{i_t}^{m_{i_t}}), 
	\end{equation}
	where 
	\[
	T(s_{i_1}^{m_{i_1}}\dotsc s_{i_t}^{m_{i_t}})= \underset{1\leq \alpha < \beta \leq t}\prod q_{i_\alpha i_\beta}^{-m_{i_\alpha}^+m_{i_\beta}^-}\left[(T_{i_1}^{m_{i_1}^-})^*\dotsc (T_{i_t}^{m_{i_t}^-})^*\right]\left[T_{i_1}^{m_{i_1}^+}\dotsc T_{i_t}^{m_{i_t}^+}\right]
	\]
	which is same as the map in (\ref{Map2}) corresponding to the finite tuple $(T_{i_1}, \dotsc, T_{i_t})$ in $\mathcal{T}$. It follows from Theorem \ref{DCIII} that $T$ is a positive definite function since $(T_{i_1}, \dotsc, T_{i_t})$ in $\mathcal{T}$ satisfies Brehmer's positivity condition. Following the similar arguments as in the proof of Theorem \ref{regular q-unitary dilationII} for $\widehat{\phi}_{c}$, we have that $\widehat{\phi}_\infty$ has an extension to a completely positive map on $\mathcal{A}_\infty$. Indeed, the same arguments as in the proof of Theorem \ref{regular q-unitary dilationII} work, because one has to deal over the finite sums at any step of the computations. Finally, we apply Stinespring's dilation theorem (see Theorem \ref{Stine}) to obtain a Hilbert space $\mathcal{K} \supseteq \mathcal{H}$ and a unital $*$-algebra homomorphism $\psi:\mathcal{A}_{\infty} \to \mathcal{B}(\mathcal{K})$ such that $\widehat{\phi}_{\infty}(a)=P_{\mathcal{H}}\psi(a)|_\mathcal{H}$ for every $ a \in \mathcal{A}_{\infty}$. Let $U_\alpha=\psi(s_\alpha+\mathcal{I}_{\infty})$ for  $\alpha \in \Lambda$. Consequently, we have a $q$-commuting family $\mathcal{U}=\{U_\alpha : \alpha\in \Lambda\}$ of unitaries acting on $\mathcal{K}$ such that
	\begin{equation}\label{eqn922}
		\underset{1\leq i<j \leq k}\prod q_{\alpha_{i}\alpha_{j}}^{-m_{\alpha_{i}}^+m_{\alpha_{j}}^-}\left[(T_{\alpha_{1}}^{m_{\alpha_{1}}^-})^*\dotsc (T_{\alpha_{k}}^{m_{\alpha_{k}}^-})^*\right]\left[T_{\alpha_{1}}^{m_{\alpha_{1}}^+}\dotsc T_{\alpha_{k}}^{m_{\alpha_{k}}^+}\right]=P_\mathcal{H}U_{\alpha_{1}}^{m_{\alpha_{1}}}\dotsc U_{\alpha_{k}}^{m_{\alpha_{k}}}|_\mathcal{H}
	\end{equation}
	for every $m_{\alpha_1}, \dotsc, m_{\alpha_k} \in \mathbb{Z}$ and $\alpha_1, \dotsc, \alpha_k \in \Lambda$ with $\alpha_1 \preceq \dotsc \preceq \alpha_k$. The proof is now complete. \qed 
	
	\medskip
	
	\noindent \textbf{Proof of Theorem \ref{dilation of DCV}.} Let $\mathcal{T}=\{T_\alpha : \alpha \in \Lambda\}$ be a $q$-commuting family of contractions  with $\|q\|=1$ such that $\mathcal{T}$ belongs to one of the classes as in the statement of Theorem \ref{dilation of DCV}. It follows from Lemma \ref{lem908} that in either case, $\mathcal{T}$ satisfies Brehmer's positivity condition. Indeed, the same arguments apply because one must address the finite sums at every step of the calculations. Consequently,  the desired conclusion follows from Theorem \ref{thm912}. \qed
	
	\medskip As an application of Theorems \ref{thm912} and \ref{dilation of DCV}, we  obtain dilation results for $q$-commuting family of contractions with $\|q\|=1$ and thereby generalizing some existing results from the literature.  
	
	\begin{defn}
		Let $\mathcal{T}=\{T_\alpha : \alpha \in \Lambda\}$ be a $q$-commuting family of contractions with $\|q\|=1$ acting on a Hilbert space $\mathcal{H}$. We say $\mathcal{T}$ admits a \textit{$q$-unitary dilation} if there exist a Hilbert space $\mathcal
		{K} \supseteq \HS$ and a $q$-commuting family $\mathcal{U}=\{U_\alpha : \alpha\in \Lambda\}$ of unitaries acting on $\mathcal{K}$ such that
		\begin{equation}\label{eqn924}
			T_{\alpha_{1}}^{m_{\alpha_{1}}}\dotsc T_{\alpha_{k}}^{m_{\alpha_{k}}}=P_\mathcal{H}U_{\alpha_{1}}^{m_{\alpha_{1}}}\dotsc U_{\alpha_{k}}^{m_{\alpha_{k}}}|_\mathcal{H}
		\end{equation}
		for every $m_{\alpha_1}, \dotsc, m_{\alpha_k} \in \mathbb{N}\cup \{0\}$ and $\alpha_1, \dotsc, \alpha_k \in \Lambda$ with $\alpha_1 \preceq \dotsc \preceq \alpha_k$. We say $\mathcal{T}$ admits a \textit{$q$-unitary extension} if each $U_\alpha$ is an extension of $T_\alpha$.
	\end{defn}
	
	It is evident that (\ref{eqn924}) is invariant under any permutation $\sigma$ on $\{\alpha_1, \dotsc, \alpha_k\}$ since the operators on both sides in \eqref{eqn924} follow the same $q$-intertwining relations. Consequently, \eqref{eqn924} is equivalent to saying that $	T_{\alpha_{1}}^{m_{\alpha_{1}}}\dotsc T_{\alpha_{k}}^{m_{\alpha_{k}}}=P_\mathcal{H}U_{\alpha_{1}}^{m_{\alpha_{1}}}\dotsc U_{\alpha_{k}}^{m_{\alpha_{k}}}|_\mathcal{H}$ for every $m_{\alpha_1}, \dotsc, m_{\alpha_k} \in \mathbb{N}\cup \{0\}$ and $\alpha_1, \dotsc, \alpha_k$ in $\Lambda$. Next, we have the following result which is a direct consequence of Theorem \ref{dilation of DCV}.
	
	\begin{thm}
		\label{dilation of DCVI}
		Let $\mathcal{T}=\{T_\alpha : \alpha \in \Lambda\}$ be a $q$-commuting family of contractions with $\|q\|=1$ acting on a Hilbert space $\HS$. Then $\mathcal{T}$ admits a $q$-unitary dilation in each of the cases given below:
		\begin{enumerate}
			\item $\mathcal{T}$ consists of isometries;
			\item $\mathcal{T}$ consists of doubly $q$-commuting contractions;
			\item $\mathcal{T}$ is a countable family and $\sum_{\alpha \in \Lambda} \|T_{\alpha}h\|^2 \leq \|h\|^2$ for all $h \in \HS$. 
		\end{enumerate}
	\end{thm}
	
	The authors of \cite{Barik} proved that a $q$-commuting tuple of contractions with $\|q\|=1$ dilates to a $q$-commuting tuple of isometries if it satisfies the positivity conditions as in \eqref{Szego}. We provide a generalization of this result to any $q$-commuting family.
	
	\begin{cor}\label{cor_403}
		Let $\mathcal{T}=\{T_\alpha : \alpha \in \Lambda\}$ be a $q$-commuting family of contractions with $\|q\|=1$ acting on a Hilbert space $\HS$ such that
		\[
		\underset{\{\alpha_1, \dotsc, \alpha_k\} \subset u}{\sum}(-1)^k(T_{\alpha_1}\dotsc T_{\alpha_k})(T_{\alpha_1}\dotsc T_{\alpha_k})^* \geq 0
		\]
		for every finite subset $u$ of $\Lambda$. Then $\mathcal{T}$ admits a $q$-commuting unitary dilation.
	\end{cor}
	
	\begin{proof}
		Note that $T_\alpha T_\beta=q_{\alpha \beta}T_\beta T_\alpha$ if and only if $T_\alpha^* T_\beta^*=q_{\alpha \beta}T_\beta^* T_\alpha^*$. Consequently, the family $\mathcal{T}^*=\{T_\alpha^* : \alpha \in \Lambda \}$ is also a $q$-commuting family of contractions with $\|q\|=1$ that satisfies Brehmer's positivity condition. The desired conclusion now follows from Theorem \ref{thm912}.
	\end{proof}	
	
	It is well-known that any commuting family of isometries admits a simultaneous unitary extension. For a proof, one may refer to Proposition 6.2 in CH-I of \cite{Nagy}. An analogous result for a doubly $q$-commuting and $q$-commuting tuple of isometries with $\|q\|=1$ was established in Theorem 6.2 of \cite{Jeu} and Theorem 6.1 of \cite{Ball}, respectively. We generalize this result to an arbitrary family of $q$-commuting isometries with $\|q\|=1$ and conclude the article.
	
	\begin{cor}\label{cor_q_iso_ext}
		Let $\mathcal{V}=\{V_\alpha : \alpha \in \Lambda\}$ be a $q$-commuting family of isometries with $\|q\|=1$ acting on a Hilbert space $\HS$. Then $\mathcal{V}$ admits a $q$-unitary extension.
	\end{cor}
	
	\begin{proof}
		We have by Theorem \ref{dilation of DCVI} that there is a Hilbert space $\mathcal{K} \supseteq \HS$ and a $q$-commuting family $\mathcal{U}=\{U_\alpha : \alpha\in \Lambda\}$ of unitaries acting on $\mathcal{K}$ such that $V_\alpha=P_\HS U_\alpha|_\HS$ for every $\alpha \in \Lambda$.  Since $V_\alpha$ and $U_\alpha$ both are isometries, we have that $V_\alpha=U_\alpha|_\HS$ for every $\alpha \in \Lambda$. The proof is complete. 
	\end{proof}
	
	\section{Declarations}
	
	\noindent (1) Data sharing is not applicable to this article as no datasets were
	generated or analysed during the current study.

	\smallskip
	
	\noindent (2) There are no competing interests.
	
	\vspace{0.2cm}
	
	\noindent \textbf{Acknowledgement.} The first named author was	supported in part by the ``Core Research Grant" of Anusandhan National Research Foundation (ANRF) with Grant No. CRG/2023/005223 of Govt. of India. The second named author was supported by the Ph.D. Fellowship of Council of Scientific and Industrial Research (CSIR), Govt. of India. The third named author was supported by the Prime Minister's Research Fellowship (PMRF ID 1300140), Govt. of India.

\end{document}